\documentclass[a4paper,11pt]{article}
\usepackage{makeidx}
\usepackage[french,english]{babel}
\usepackage[latin1]{inputenc}
\usepackage{amsfonts}
\usepackage{bm}
\usepackage{amssymb}
\usepackage{latexsym}
\usepackage{amsmath}
\usepackage{amscd}
\usepackage{amsthm}
\usepackage{hyperref}
\usepackage{rotating}
\usepackage{graphicx}
\usepackage{pifont}
\usepackage{curves}
\usepackage{fancyhdr}
\usepackage{epsfig}
\usepackage{pstricks}
\usepackage{pst-tree}
\usepackage{subfigure}
\usepackage{epic}
\usepackage{alltt}
\usepackage[all]{xy}
\usepackage{appendix}

\theoremstyle{plain}
\newtheorem{theorem}{Theorem}[section]
\newtheorem{prop}[theorem]{Proposition}
\newtheorem{lemme}[theorem]{Lemma}

\theoremstyle{definition}
\newtheorem{definition}[theorem]{Definition}

\newtheorem{ex}[theorem]{Example}
\newtheorem{rque}[theorem]{Remark}

\title{Nonlinear historical superprocess approximations for population models with past dependence}

\author{Sylvie Méléard\thanks{CMAP, Ecole Polytechnique, UMR CNRS 7641, Route de Saclay, 91128 Palaiseau Cédex, France} ~ and ~ Viet Chi Tran\thanks{Equipe Probabilité Statistique, Laboratoire Paul Painlevé, UMR CNRS 8524, UFR de Mathématiques, Université des Sciences et Technologies Lille 1, Cité Scientifique, 59655 Villeneuve d'Ascq Cédex, France ; CMAP, Ecole Polytechnique}}

\date{\today}

\numberwithin{equation}{section}

\newcommand{\Co}{\mathcal{C}}

\newcommand{\Cov}{\mbox{Cov}}

\def\Id{\mathbf{Id}}
\def\D{\mathbb{D}}

\def\H{\mathcal{H}}
\def\Q{\mathbb{Q}}

\def\N{\mathbb{N}}

\def\P{\mathbb{P}}
\def\R{\mathbb{R}}
\def\E{\mathbb{E}}

\def\dSk{\mathbf{d}_{\mbox{{\scriptsize Sk}}}}

\def\ind{{\mathchoice {\rm 1\mskip-4mu l} {\rm 1\mskip-4mu l}
{\rm 1\mskip-4.5mu l} {\rm 1\mskip-5mu l}}}

\def\eg{\textit{e.g.} }
\def\ie{\textit{i.e.} }

\def\etal{\textit{et al.} }

\begin{document}

\maketitle

\begin{abstract}
We are interested in the evolving genealogy of a birth and death process with trait structure and ecological interactions. Traits are hereditarily transmitted from a parent to its offspring unless a mutation occurs. The dynamics may depend on the trait of the ancestors and on its past and allows interactions between individuals through their lineages. We define an interacting historical particle process  describing the  genealogies of the living individuals; it takes values in the space of point measures  on an infinite dimensional càdlàg path space. This individual-based process can be approximated by  a nonlinear historical superprocess, under the assumptions of large populations, small individuals and allometric demographies. Because of the interactions, the branching property fails and we use martingale problems and fine couplings between our population and independent branching particles. Our convergence theorem is illustrated by two examples of current interest in biology. The first one relates the biodiversity history of a population and its phylogeny, while the second treats a spatial model with competition between individuals through their past trajectories.\end{abstract}


\section{Introduction}

\par The evolution of genealogies in population dynamics is a major problem, which motivated an abundant literature and has applications to evolution and population genetics. Our purpose here is to generalize the existing models by emphasizing the ecological interactions, namely the competition between individuals for limited resources. In this paper, we construct a structured birth and death process with mutation and competition, whose dynamics depends on the past.
Each individual is characterized by a vector trait $x\in \R^d$, which remains constant during the individual's life and is transmitted hereditarily unless a mutation occurs. The birth and death rates of an individual can depend on the traits of its ancestors or on the trait's age in the lineage. Moreover individuals interact with each other. \\
Here, we are interested in keeping track of the genealogies of individuals with small weights, in large populations with allometric demographies (i.e. short individual lives and reproduction times).
At a given time $t$, we associate to each individual its lineage until $t$ defined as a function which associates to each $s\leq t$ the trait of its ancestor living at $s$ (see Fig. \ref{fig1}). Since each individual keeps a constant trait during its life, these lineages are càdlàg piecewise constant paths. The population is described by a point measure on the lineage space.
To reflect allometric demographies, we introduce a parameter $n$ which scales the size of the population. The individuals are weighted by $1/n$ to keep the total biomass of constant order when $n$ varies. Moreover, lifetimes and gestation lengths are proportional to this weight. Hence birth and death rates are of order $n$ and preserve the demographic balance. Mutation steps are assumed to have a variance of order $1/n$. The trait evolution is driven by mutations and competition between individuals.\\

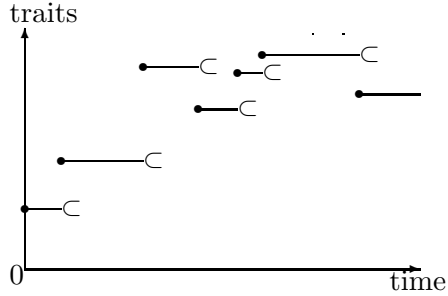
\begin{figure}[ht]
\begin{center}
    \unitlength=0.4cm
    \begin{picture}(15,10)
      \put(1,1){\vector(1,0){13}} \put(1,1){\vector(0,1){8}}
      \put(0.5,0.5){0}
      \put(0.5,9.2){traits}
      \put(13,0.3){time}
      \put(1,3){\line(1,0){1.2}}\put(1,3){\circle*{0.25}}\put(2.2,2.8){{\small$\subset$}}
      \put(2.2,4.6){\line(1,0){2.7}}\put(2.2,4.6){\circle*{0.25}}\put(4.9,4.4){{\small$\subset$}\dottedline(4.9,4.4)(5.9,4.4)}
      \put(4.9,7.7){\line(1,0){1.8}}\put(4.9,7.7){\circle*{0.25}}\put(6.7,7.5){{\small$\subset$}}
      \put(6.7,6.3){\line(1,0){1.3}}\put(6.7,6.3){\circle*{0.25}}\put(8,6.1){{\small$\subset$}}
      \put(8,7.5){\line(1,0){0.8}}\put(8,7.5){\circle*{0.25}}\put(8.8,7.3){{\small$\subset$}}
      \put(8.8,8.1){\line(1,0){3.2}}\put(8.8,8.1){\circle*{0.25}}\put(12,7.9){{\small$\subset$}}
      \put(12,6.8){\line(1,0){2}}\put(12,6.8){\circle*{0.25}}
    \end{picture}
  \end{center}
  \vspace{-0.5cm}
  \caption{{\small\textit{Example of a path constituted with ancestral traits. 
  }}}\label{fig1}
\end{figure}

\noindent We study the convergence of these processes in large populations, when $n$ tends to infinity. We proceed with tightness-uniqueness arguments inspired by  Dawson and Perkins \cite{dawsonperkinsAMS} for the historical super-Brownian process (without interaction). In our model, there is no branching property, impeding the use of Laplace's transform arguments. We introduce a new infinite dimensional martingale problem and use fine comparisons between our population and independent branching particles. Moreover, the measure-valued processes on path space that we study are discontinuous, which makes it necessary to define a new class of test functions.
In the limit, we obtain a nonlinear historical superprocess, generalizing the work of Perkins \cite{perkinsAMS}. In particular, past dependence and competition can be taken into account in the lineages. Examples are given, suggesting that the historical processes may open the way to models of evolution without the assumption of rare mutations and time scale separations. Let us remark that our model allows the description of both genealogies and population densities in the forward physical time, with variable population size including extinction phenomena. This can be a first step to reconstruct the past biodiversity from the phylogenies of living species with ecological interactions.

\medskip \noindent From this convergence result, we study the dynamics of the distribution of genealogies. We show  that it can be seen as the dynamics of a diffusive particle system resampled in such a way that individuals with larger allometric (diffusive) coefficient or with higher growth rate are more likely to be chosen.  A particular case is when these two rates are constant. Then, the intensity of the distribution of the genealogies corresponds to a Fleming-Viot process with genealogies, similar to the one introduced in  Greven-Pfaffelhuber-Winter \cite{grevenpfaffelhuberwinter}.  In the case without interaction or with logistic competition without past or trait dependence, one can also obtain Feynman-Kac's formulas for the mass process renormalized by its expectation. This formula can help in getting easier simulation schemes for certain demographic quantities.
 Let us note that more generally, in absence of interaction, coalescent processes provide another modeling of genealogies (see Berestycki \cite{berestycki} and references therein for a survey). Depperschmidt, Greven, Pfaffelhuber and Winter \cite{depperschmidtgrevenpfaffelhuber, grevenpfaffelhuberwinter} represent genealogical trees (for instance $\Lambda$-coalescents) as marked ultra-metric spaces and the absence of interaction allow them to study genealogical distances by using Laplace transforms. In a recent work, Barton, Etheridge and Véber \cite{bartonetheridgeveber} study the genealogies of a spatial version of $\Lambda$-Fleming-Viot processes and of their various limits. All these models allow the incorporation of selection and mutation (see also \eg \cite{kroneneuhauser, bartonetheridge,etheridgepfaffelhuberwakolbinger}) but not of competition between individuals, which is the main aim of our work. \\


\noindent  In  Section 2, we construct the historical particle system whose dynamics depends on the past. The diffusive limit in large population is obtained in Section 3 under large population and allometric demography assumptions. We prove tightness of the laws of the historical particle processes  with new arguments and test functions  accounting for interactions and jumps.
We then identify the limiting values as solutions of a nonlinear martingale problem for which uniqueness is stated. The distribution of genealogies is studied in Section 4 and examples are carried in Section 5. The first example deals with an evolution model of adaptive dynamics with local (see \cite{dieckmanndoebeli}) competition. The historical superprocess dynamics  shows that for a range of parameters, the population separates into groups concentrated around some trait values. This cannot be observed by the unique information on the (classical) superprocess as noticed in Figure  \ref{DD1}.
This example suggests that the historical processes may allow one to understand evolution  without the assumption of rare mutations and time scale separation.  For example, our model could be a basis to reconstruct the past biodiversity from the actual molecular phylogenies. Until now the models which have been used  do not take into account the interaction between individuals (see Morlon et al. \cite{morlon}).
The second example is a spatial model (see \cite{adlertribe} or \cite[p.50]{perkinsAMS}): particles consume resources where they are locating. Therefore, the offspring arriving in previously habited regions are penalized. We will see that this tends to separate the cloud of particles in several distinct families whose common ancestor is very old.

\bigskip \noindent
\textbf{Notation:}
For a given metric space $E$, we denote by $\mathbb{D}_E=\mathbb{D}(\R_+,E)$ the space of càdlàg functions from $\R_+$ to $E$. For $E=\R$, we will use the more simple notation $\mathbb{D}=\mathbb{D}(\R_+,\R)$. These spaces are embedded with the Skorohod topology  which makes them Polish spaces. (\eg \cite{billingsley,jakubowski,joffemetivier}, see also \eqref{dsk} in appendix).

\noindent For a function $x\in \mathbb{D}_E$ and $t>0$, we denote by $x^t$ the stopped function defined by $x^t(s)=x(s\wedge t)$ and by $x^{t-}$ the function defined by $x^{t-}(s) = \lim_{r\uparrow t} x^r(s)$. We will also often write $x_t=x(t)$ for the value of the function at time $t$.
For $y,w\in \mathbb{D}_{E}$ and $t\in \R_+$, we denote by $(y|t|w)\in \mathbb{D}_{E}$ the following path:
\begin{equation}
(y|t|w)=\left\{\begin{array}{ccc}
y_u & \mbox{ if } & u<t\\
w_{u-t} & \mbox{ if } & u\geq t.
\end{array}\right.
\end{equation}
When the path $w$ is constant with $\forall u\in \R_+,\,w_u=x$, we will write $(y|t|x)$ with a notational abuse.

\medskip \noindent
 We denote by $\mathcal{M}_F(E)$ (resp. $\mathcal{M}_P^n(E)$, $\mathcal{P}(E)$) the set of finite measures on $E$ (resp. of point measures renormalized by $1/n$, of probability measures). These spaces are embedded with the topology of weak convergence. 

\section{The historical particle system}

In this section, we construct the finite interacting historical particle system that we will study. Trait-structured particle systems without dependence on the past have been considered in Fournier and Méléard \cite{fourniermeleard} or Champagnat \etal \cite{champagnatferrieremeleard}. For populations with age-structure, we refer to Jagers \cite{jagerslivre,jagers} and Méléard and Tran \cite{meleardtran,trangdesdev} for instance. Here, we are inspired by these works and propose a birth and death particle system where the lineage of each particle, \ie the traits of its ancestors, is encoded into a path of $\D_{\R^d}$.

\subsection{Lineage}

We consider a discrete population in continuous time where the individuals reproduce asexually and die with rates that depend on a
hereditary trait and on their past. Each individual is associated with a quantitative trait transmitted from its parent except when a mutation occurs.
The rates may express through the traits carried by the ancestors of the individual. One purpose is for example to take into account the accumulation of beneficial
and deleterious mutations through generations.\\

\noindent Individuals are characterized by a trait $x\in \R^d$. The lineage or past history of an individual is defined by the succession of ancestral traits with their appearance times and by the succession of ancestral
reproduction times (birth of new individuals). To an individual of trait $x$ born at time $S_m$, having $m-1$ ancestors born after $t=0$ at times $S_1=0<S_2<\dots <S_{m-1}$, with $S_{m-1}<S_m$, and of traits $(x_1,x_2,\dots x_{m-1})$, we associate the path
\begin{equation}
y_t=\sum_{j=1}^{m-1} x_{j} \mathbf{1}_{S_{j}\leq t<S_{j+1}}+x \mathbf{1}_{S_m\leq t}.\label{deftrajectoire}
\end{equation}
This path is called the \textit{lineage} of the individual. We denote by $\mathcal{L}$ the set of possible lineages of the form \eqref{deftrajectoire}. Since a path in $\mathcal{L}$ is entirely characterized by the integer $m$ and the sequence $(0,x_1,\dots,s_{m-1},x_{m-1},s_m,x)$ of jump times and traits, it is possible to describe each element of $\mathcal{L}$ by an element of $\N\times \bigcup_{m\in \N} (\R_+\times \R^d)^m$ embedded with a natural lexicographical order.

\subsection{Population dynamics}\label{sectiondescriptionsuperp}

Let us introduce a parameter $n\in \N^*=\{1,2,\dots\}$ scaling  the carrying capacity, when the total amount of resources is fixed. To keep the total biomass constant, individuals are attributed a weight  $1/n$. The population is represented by a point measure as follows:
\begin{equation}
X^n_t := \frac{1}{n}\sum_{i=1}^{N^n_t}\delta_{y^i_{.\wedge t}}\; \in \; \mathcal{M}^n_P(\mathcal{L}) \subset \mathcal{M}^n_P(\D_{\R^d}),\label{defht}
\end{equation}where $N^n_t=n\, \langle X^n_t,1\rangle$ is the number of individuals alive at time $t$.\\




\noindent  The reproduction is asexual and the offsprings inherit  the trait of the ancestor except when a mutation occurs.  Death can be due to the background of each individual or  to competition with the other individuals.  We consider allometric demographies where lifetimes and gestation lengths of individuals are proportional to the biomass. Thus, birth and death rates are of order $n$, but respect the constraint of preservation of the demographic balance. Also, the mutation steps are rescaled by $1/n$.

\noindent Let us now define the population dynamics. For $n\in \N^*$, we consider an individual characterized at time $t$ by the lineage $y\in \D_{\R^d}$ in a population $X^n\in \D_{\mathcal{M}_P^n(\D_{\R^d})}$.\\

\noindent \textbf{Reproduction:} The birth rate at time $t$ is $b^n(t,y)$, where
$$ b^n(t,y)= n\,r(t,y)+ b(t,y).$$
The function $b$ is a continuous nonnegative real function on $\R_+\times \D_{\R^d}$. For instance, $b$ can be chosen in the form:
\begin{align}
 b(t,y)= & B\Big(\int_0^t y_{t-s} \nu_b(ds)\Big)\label{birthrate2}
 \end{align} where $B$ is a continuous function bounded by $\bar{B}$ and $\nu_b$ is a Radon measure on $\R^+$.  The allometric function $r$ is given by:
\begin{equation}
r(t,y)=R\Big(\int_0^t y_{t-s}\nu_r(ds)\Big)\label{def:R}
\end{equation}where $R$ is continuous and bounded below and above by $\underline{R}>0$ and $\bar{R}>0$, and where $\nu_r$ is a Radon measure on $\R_+$. We also assume that $R$ is Lipschitz continuous, implying that $R^{1/2}$ is also Lipschitz.

\medskip \noindent
When an individual with trait $y_{t_-}$ gives birth at time $t$, the new offspring is either a mutant or a clone.  With probability $1-p\in [0,1]$,
the new individual is a clone of its parent, with same trait $y_{t_-}$
and same lineage $y$.  With probability $p\in [0,1]$, the offspring is a mutant of trait
$y_{t_-}+h$, where $h$ is drawn in the distribution $k^n(h)\,dh$. To  this mutant is associated the lineage $(y|t|y_{t_-}+h)$.
For the sake of simplicity,  the mutation density $k^n(h)$ is assumed to be a Gaussian density with mean $0$ and covariance matrix $\sigma^2 \, \Id/n$. However the model could be generalized for instance to mutation densities $k^n(y_{t_-},h)$ with dependence on the parent's trait. Let us  introduce the notation:
\begin{equation}
K^n(dh)=p k^n(h)dh+(1-p)\delta_0(dh).\label{def_K^n}
\end{equation}

\begin{ex}
(i) If  $\nu_b(ds)=\delta_0(ds)$, then $\int_0^t y_{t-s} \nu_b(ds)=y_t$ is the trait of the individual of the lineage $y$ living at time $t$.\\
(ii) If  $\nu_b(ds)=e^{-\alpha s}ds$, with $\alpha>0$, then $\int_0^t y_{t-s} \nu_b(ds)=\int_0^t e^{-\alpha(t-s)}y_s ds$. This means that the traits of recent ancestors have a higher contribution in the birth rate of the individual alive at time $t$. Such rates may be useful to model social interactions, for instance cooperative breeding where the ancestors contribute to protect and raise their descendants. When ancestors have advantageous traits, they may help their offspring to reproduce in more favorable conditions and increase their birth rates.\hfill $\Box$
\end{ex}

\noindent \textbf{Death:} To define the death rate, let us consider a bounded continuous
interaction kernel $U\in \Co_b(\R_+\times \D_{\R^d}^2,\R)$, a bounded continuous function $D$ on $\R_+\times \D_{\R^d}$ and a Radon measure $\nu_d$ weighting the influence of the past population on the present individual $y$ at time $t$. The death rate is
$$d^n(t,y,X^n)=  n\,r(t,y)+d(t,y,X^n),$$
where for a process $X\in \D_{\mathcal{M}_F(\D_{\R^d})}$,
\begin{align}
 d(t,y,X)= & D(t,y)+ \int_0^t \int_{\D_{\R^d}} U(t,y,y') X_{t-s}(dy') \nu_d(ds).\label{deathrate2}
 \end{align} The first term with function $r$ allows us to preserve the demographic balance. The  term $D(t,y)$ is the natural death rate, while $U(t,y,y')$ represents the competition exerted at time $t$ on the  individual of lineage $y$, by an individual of lineage $y'$ alive in its past.
We assume that:
\begin{align}
& \exists \bar{D}>0,\, \forall y\in \D,\,\forall t\in \R_+,\,
 0\leq D(t,y)<\bar{D},\nonumber\\
&\exists \underline{U},\,\bar{U}>0,\, \forall y,y'\in \D,\,\forall t\in \R_+,\,
 0<U(t,y,y')<\bar{U}.\label{hyptaux}
 \end{align}

\begin{ex}\label{ex2.2}Model with asymmetrical competition:\\
 If we choose $D=0$, $\nu_d(ds)=\delta_0(ds)$ and the asymmetric competition kernel proposed by Kisdi \cite{kisdi}
\begin{equation}
U(t,y,y')= \frac{2}{K} \Big(1-\frac{1}{1+\alpha e^{-\beta(y_t-y'_t)}}\Big), \label{noyaukisdi}
\end{equation}then, the death rate becomes:
\begin{equation}
 d(t,y,X)=\frac{2}{K} \int_{\D_{\R^d}}\Big(1-\frac{1}{1+\alpha e^{-\beta(y_t-y'_t)}}\Big)X_t(dy'). \label{tauxmortkisdi}
\end{equation}
\end{ex}

\begin{rque}\label{rque2.3} A generalization to physical age structure as considered in \cite{meleardtransuperage} is possible, provided we extend the trait space by giving a color to each individual. The colors are independent uniform $[0,1]$-valued random variables  drawn at each birth. The lineage of colors of an individual is a path $c\in \mathbb{D}(\R_+,[0,1])$ and allows the definition of the
 birth date of an individual  alive at time $t$,
\begin{eqnarray}
\tau_{c,t} & =  \inf\{s\leq t,\quad c_s= c_{t}\} & = \sup\{s\leq t,\quad c_s\not= c_{t}\}\label{tauct}
\end{eqnarray}
The age of the latter individual at time $t$ is given by
\begin{equation}
a(t):=t-\tau_{c,t}.\label{definitionage}
\end{equation}
It is a càdlàg function, discontinuous at the birth times. \hfill $\Box$
\end{rque}

\subsection{Construction of the historical particle process}


Following the work of Fournier and Méléard \cite{fourniermeleard}, the population process is obtained as a solution of a stochastic differential equation (SDE) driven by a Poisson Point Measure (PPM),  describing  the dynamics of $(X^n_t)_{t\in \R_+}$, for any $n\in \N^*$. The measure representing the population evolves through the occurrences of births and deaths. Since the rates may vary with time,  acceptance-rejection techniques are used to obtain these events' occurrences by mean of PPMs. According to birth and death events,   Dirac masses are added  or removed. This construction provides an exact simulation algorithm that is extensively used in Section \ref{sectionexemples}.

\begin{rque} As in Fournier and Méléard \cite{fourniermeleard},  the map $Y=(Y^i)_{i\in \N^*}$ from $\bigcup_{n\in \N^*}\mathcal{M}^n_P(\mathcal{L})$ in $\mathcal{L}$  is defined for $n$ and $N$ in $\N^*$ by \begin{align*}
Y^{j}\left(\frac{1}{n}\sum_{i=1}^{N} \delta_{y^i}\right)  =  \left\{\begin{array}{cc}y^j, & \mbox{ if }j\leq N\\
0 & \mbox{ otherwise},\end{array}\right.
\end{align*}
where the individuals are sorted by the lexicographical order. This will be useful to extract a particular individual from
the population. When there is no ambiguity, we will write $Y^i$ instead of $Y^i(X)$ for a point measure $X$.
\end{rque}

\bigskip

\begin{definition}\label{defprocmicrosuperp}Let us consider on the probability space $(\Omega,\mathcal{F},\mathbb{P})$:
\begin{enumerate}
\item a random variable $X^n_0\in \mathcal{M}_P^n(\mathbb{D}_{\R^d})$ satisfying $\mathbb{E}\left(\langle X^n_0,1\rangle \right)<+\infty$ and such that the support of $X^n_0$ contains a.s. only constant functions,
\item a PPM $Q(ds,di,dh,d\theta)$ on $\R_+\times E:=\R_+\times \N^*\times \R^d \times \R_+$ with intensity measure $ds\otimes n(di)\otimes dh\otimes  d\theta$ and independent from $X^n_0$, $n(di)$ being the counting measure on $\N^*$ and $ds$, $dh$ and $d\theta$ the Lebesgue measures on $\R_+$, $\R^d$ and $\R_+$ respectively.
\end{enumerate}
We denote by $(\mathcal{F}^n_t)_{t\in \R_+}$ the canonical filtration associated with $X^n_0$ and $Q$, and consider the following SDE with values in $\mathcal{M}^n_P(\mathbb{D}_{\R^d})$:
\begin{align}
 X^n_t   =     X^n_0+ & \frac{1}{n}\int_0^t \int_{E} \mathbf{1}_{\{i\leq n \langle X^n_{s_-},1\rangle\}}\left[ \delta_{(Y^i|s|Y^i_{s_-}+h)}\ind_{\theta\leq m^n_1(i,s,h)}+\delta_{Y^i}\ind_{m^n_1(i,s,h)<\theta\leq m^n_2(i,s,h)}\right.\nonumber\\
 - & \left.\delta_{Y^i}\ind_{m^n_2(i,s,h)<\theta\leq m^n_3(i,s,h,X^{n,s_-})}\right] Q(ds, di,dh,d\theta),
\label{descriptionalgosuperprenormalisen}
\end{align}
where 
\begin{align}
& m^n_1(i,s,h)  =  p\,b^n(s,Y^i(X^n_{{s_-}}))k^n(h),\nonumber\\
& m^n_2(i,s,h)  =  m^n_1(i,s,h)+(1-p)\,b^n(s,Y^i(X^n_{s_-}))k^n(h)\nonumber\\
& m^n_3(i,s,h,X^{n,s_-}) =  m^n_2(i,s,h)+ d^n(s,Y^i(X^n_{s_-}),X^{n,s_-})k^n(h).\label{defm1m2n}
\end{align}
\hfill $\Box$
\end{definition}



\noindent Existence and uniqueness of the solution $(X^n_t)_{t\in \R_+}$ of SDE \eqref{descriptionalgosuperprenormalisen} are obtained for every $n\in \N^*$, from a direct adaptation of  \cite{fourniermeleard}, as well as the next proposition concerning moment estimates and martingale properties.

\begin{prop}\label{prop_usefulresults} Let us assume  that
\begin{align}
\label{uniform}
& \sup_{n\in \N^*}\mathbb{E}\left(\langle X^n_0,1\rangle^3\right)<+\infty.
\end{align}
Then,\\
\noindent (i)
For all $T>0$,
\begin{align}
\sup_{n\in \N^*}\E\Big(\sup_{t\in [0,T]}\langle X^n_t,1\rangle^3\Big)<+\infty.\label{estimee_momentp}
\end{align}
(ii) For a bounded and measurable function $\varphi$,
\begin{align}
\langle X^n_t,\varphi\rangle
=    \langle X^n_0,\varphi\rangle + &  M^{n,\varphi}_t
+    \int_0^t ds \int_{\D_{\R^d}}  X^n_s(dy)   \Big[ \nonumber\\
 & n r(s,y)\Big(\int_{\R^d} \varphi(y|s|y_s+h) K^n(y_s,dh)-\varphi(y)\Big) \nonumber\\
+&  b(s,y) \int_{\R^d} \varphi(y|s|y_s+h)K^n(y_s,dh) -d(s,y,(X^n)^s)\varphi(y)\Big]\label{PBM}
\end{align}where $M^{n,\varphi}$ is a square integrable martingale starting from 0 with quadratic variation
\begin{multline}
\langle M^{n,\varphi}\rangle_t=  \frac{1}{n}\int_0^t \int_{\D_{\R^d}}  \Big[
\big(n r(s,y)+b(s,y)\big) \int_{\R^d} \varphi^2(y|s|y_s+h) K^n(y_s,dh) \\
+   \big(n r(s,y)+d(s,y,(X^n)^s)\big)\varphi^2(y)\Big]X^n_s(dy)\,ds.\label{crochetPBM}
\end{multline}
\end{prop}

\subsection{Examples}\label{sectionexemples}

In this section, we give two examples of applications of historical processes in biology.



\subsubsection{A model with competition for resources}\label{section:example1}

We first introduce a model of adaptation with competition for resources that has been considered by Roughgarden \cite{roughgarden}, Dieckmann and Doebeli \cite{dieckmanndoebeli}, Champagnat and Méléard \cite{champagnatmeleard2011}. In this model, the trait $x\in [0,4]$ can be thought as being a (beak) size. The birth and death rates are chosen as
\begin{align*}
r(t,y)=1,\qquad b(t,y)=\exp\Big(-\frac{(y_t-2)^2}{2\sigma^2_b}\Big),\qquad d(t,y,X)=\int_{\D} \exp\Big(-\frac{(y_t-y'_t)^2}{2\sigma_U^2}\Big)X(dy').
\end{align*}
Here, there is no dependence on the past and $\gamma(s,y,\bar{X}^s)=\gamma(y_s,\bar{X}_s)$. Non-historical superprocess renormalizations in such models have been investigated for instance by Champagnat et al. \cite{champagnatferrieremeleard_stochmodels}.
 The birth rate is maximal at $x^*=2$ and there is a local competition with neighbors of close traits. A flat  competition kernel  ($\sigma_U=+\infty$) will make evolution  favoring  individuals with maximal growth rate $x^*$. For $\sigma_U<+\infty$, Champagnat and Méléard \cite{champagnatmeleard2011} proved that, depending whether $\sigma_b<\sigma_U$ or $\sigma_b>\sigma_U$, either evolution concentrates on individuals with trait $2$ or favours several groups concentrated around different trait values.

\begin{figure}[!ht]
\begin{center}
\begin{tabular}{ccc}
(a) & \includegraphics[width=5.5cm,angle=270]{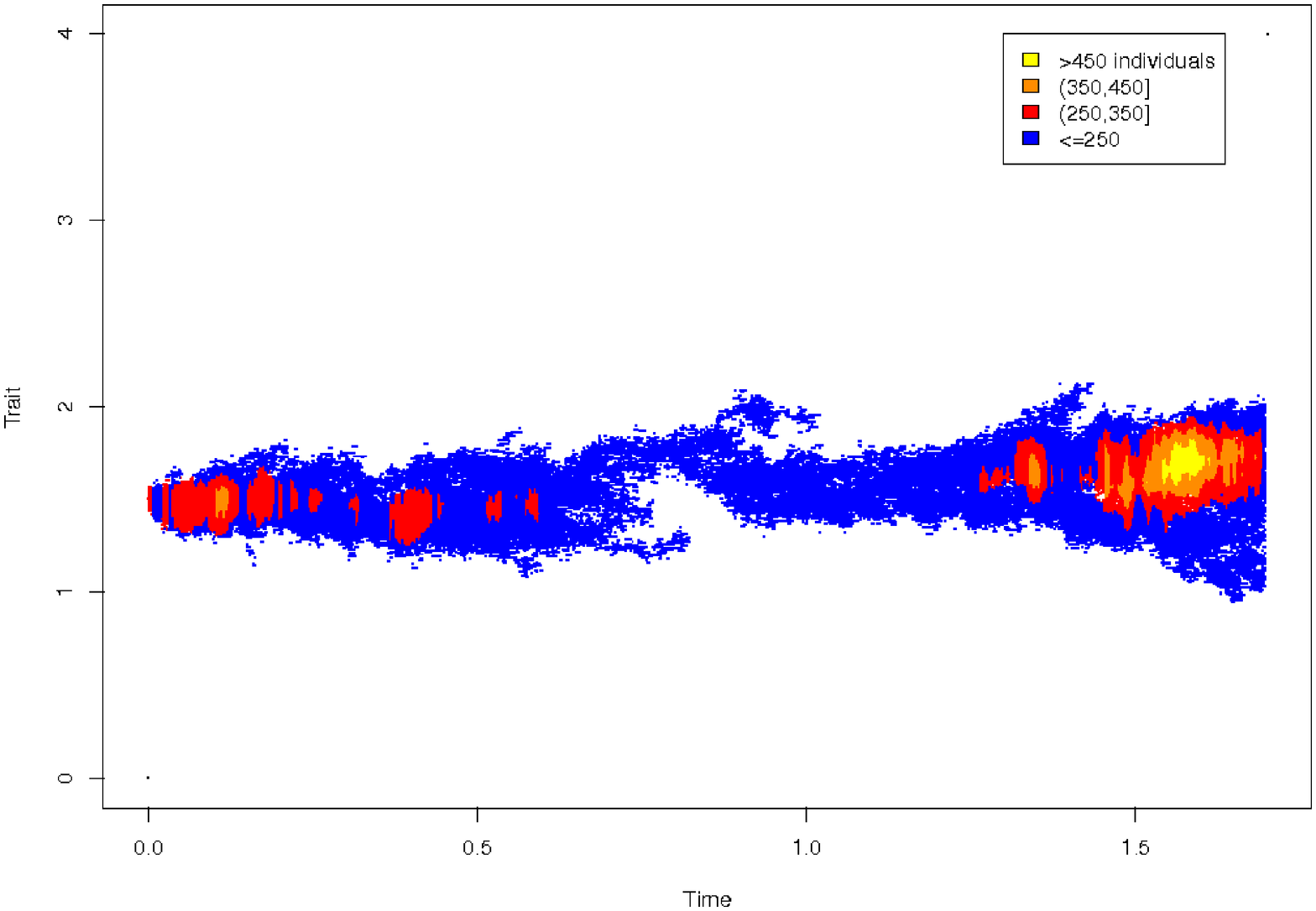} &     \includegraphics[width=5.5cm,angle=270]{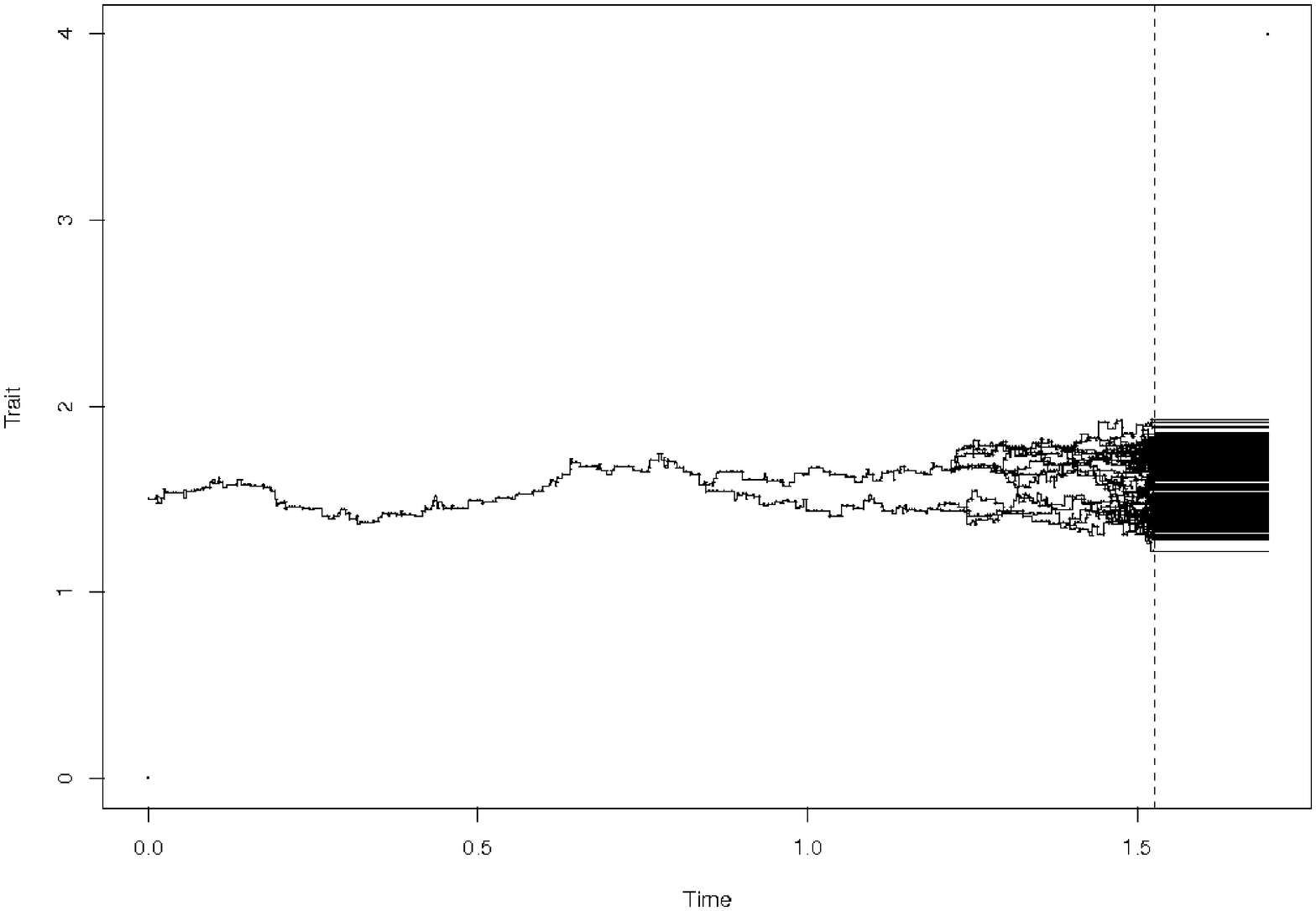}\\
(b) & \includegraphics[width=5.5cm,angle=270]{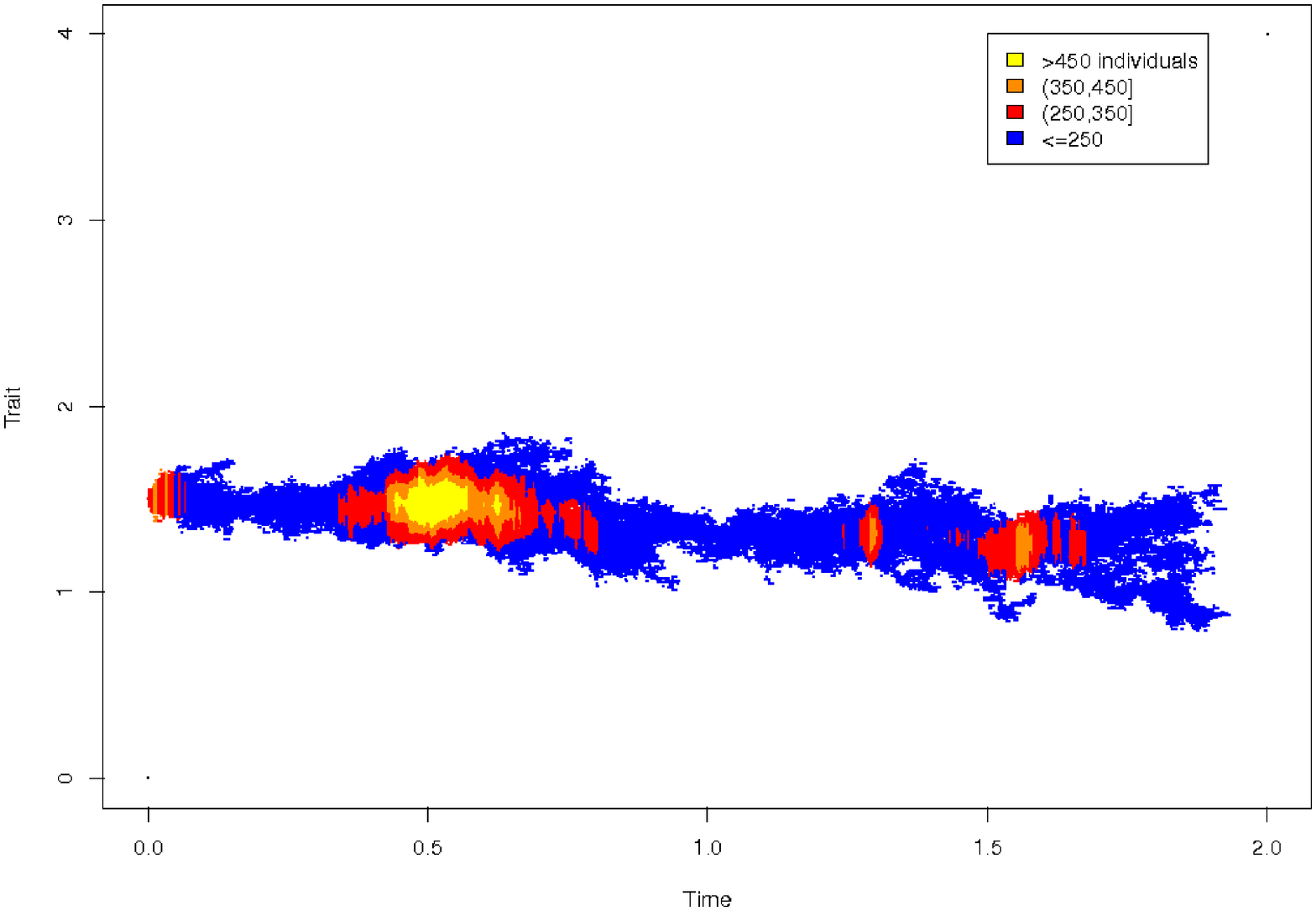} &     \includegraphics[width=5.5cm,angle=270]{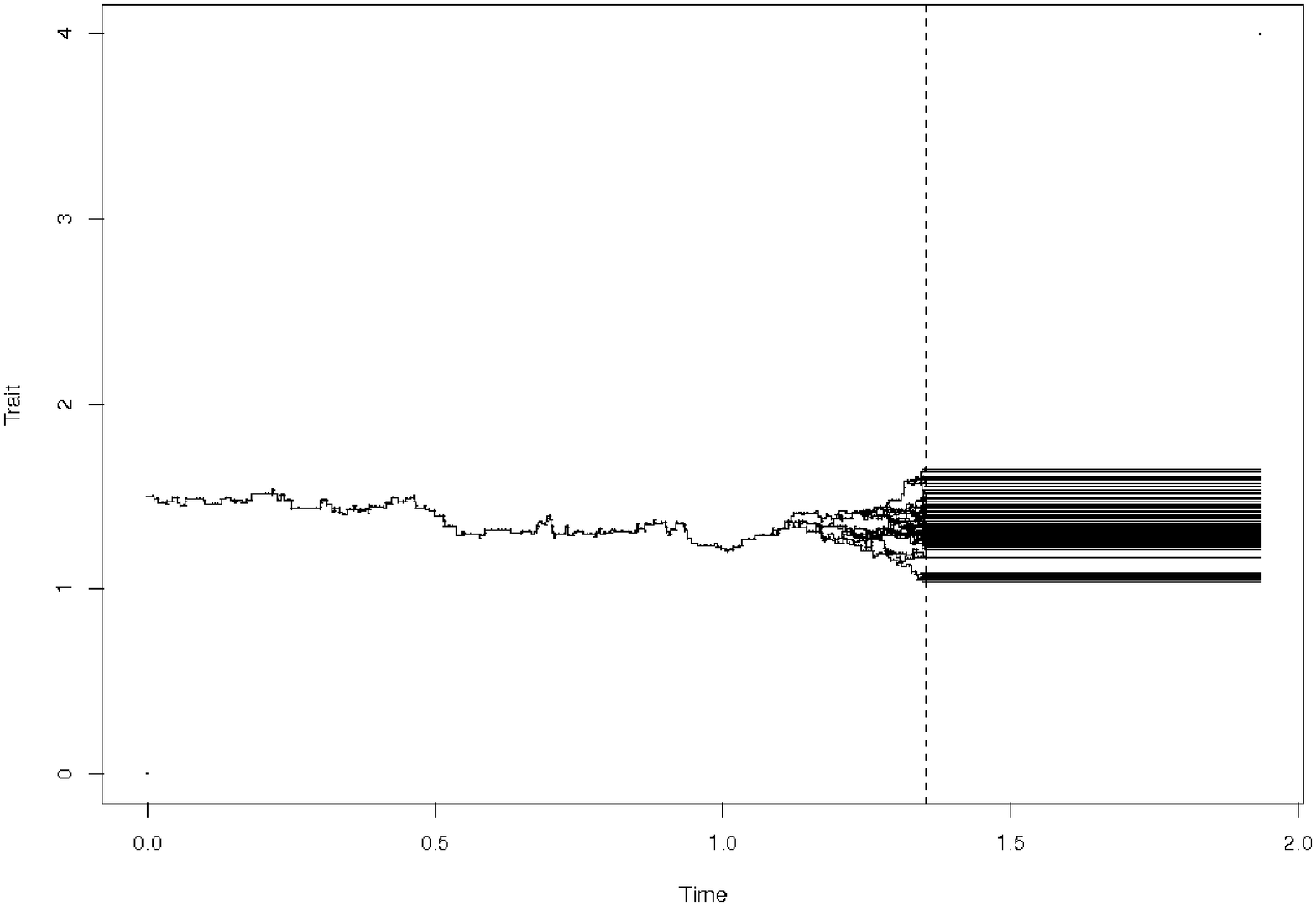}\\
     \end{tabular}
\caption{{\small \textit{Dieckmann-Doebeli model. $\sigma=0.4$, $\sigma_b=0.4$, $x_0=4$, $p=0.5$, $n=300$. The 300 particles are started with the trait $1.5$. (a) $\sigma_U=0.3$ ; (b) 
$\sigma_U=0.45$. When $\sigma_b>\sigma_U$, we observe in the historical superprocess a separation of the population into subgroups concentrated around different trait values and attached to the MRCA by a single branch. This genealogical separation of the population into a small number of families (2 in (a)) cannot be seen on the superprocesses (left column). When $\sigma_b<\sigma_U$, then the families are less distinct and the population has a shorter persistence. }}}\label{DD1}
\end{center}
\end{figure}

\medskip
\noindent In Figure \ref{DD1}, we represent the dynamics of the population process including the extinct branches and the historical particle process. These two representations complement each other. The population process gives the evolution of the number density of individuals, while the historical particle process reveals the evolving phylogenies and provides the ancestral paths of living particles. The simulations are obtained for a large $n$ and show two different behaviors. It is known that the limit of the population process is a superprocess (left pictures). Our aim is to understand the limit of the historical particle process (right pictures), i.e. the nonlinear historical superprocess.

In the superprocess, the individual dimension is lost and ancestral paths can not be read from the sole information of the support of the measure. The latter represents the dynamics of the population's biodiversity. The interest of the historical superprocess is to provide genealogical information. Indeed, one can stress that the historical process at time $t$ restricted to paths up to time $t-\varepsilon$ has finite support, as observed in Figure \ref{DD1}. This result is known for the historical super Brownian motion \cite[Section 3]{dawsonperkinsAMS}.


In the simulations of Figure \ref{DD1}, we see that when $\sigma_b>\sigma_U$, two families appear, in the sense that in the genealogical tree (on the right) the most recent common ancestor (MRCA) is separated from the next branching events by a number of generations of order $n$. Intuitively, the range of the competition kernel being smaller than the one of the mutations, the population fragments into small patches that have few interactions with each other.
When $\sigma_b<\sigma_U$, the distinction of subfamilies in the genealogical tree is not so clear. The competition kernel has a large range and prevents the new mutants from escaping the competition created by the other individuals in the population.

\subsubsection{A variant of Adler's fattened goats: a spatial model}\label{section:example2}

In many models, trait or space play a similar role. Spatial models have been extensively studied as toy models for
evolution (see Bolker and Pacala \cite{bolkerpacala, bolkerpacalaAN} or Dieckmann and Law \cite{dieckmannlaw2000}).
Here we consider a spatial model where the competition exerted by past ancestors is softened. This model is a variant of Adler's fattened goats (\eg \cite{adlertribe,perkinsAMS}), but with an interaction  in the drift term and not in the growth rate term). It corresponds to the choice of $D=0$, $\nu_d(ds)=e^{-\alpha s}ds$ with $\alpha>0$, $r(t,y)=1$, $b(t,y)=b$ and
\begin{align}
d(t,y,X)=\int_0^t \int_{\R^d}
 \frac{K_\varepsilon\big(y'(s)-y(t)\big)}{K} X_s(dy',dc')\,e^{-\alpha(t-s)}ds,\label{ex1}
\end{align}where, $K_\varepsilon$ is a symmetric smooth kernel with maximum at $0$, for instance the density function of a centered Gaussian distribution with variance $\varepsilon$. The death rate corresponds to the choice of $U(t,y,y')=K_\varepsilon\big(y_t-y'_t\big).$ From the definition of the processes (see \eqref{deftrajectoire}), if $y'$ belongs to the support of $X_s(dy')$ then almost surely (a.s.) $y'$ is a path stopped at $s$ and for all $t\geq s$, $y'_t=y'_s$.

The goat-like particles consume resources at their location. When they arrive in a region previously grazed by the population, their death rate  increases. The
parameter $\alpha$ describes the speed at which the environment replenishes itself. The kernel $K_\varepsilon$ is
 the density function of a centered Gaussian distribution with variance $\varepsilon$. The parameter $K$
 scales the  carrying capacity and controls the mass of $\bar{X}_s$.


\begin{figure}[!ht]
\begin{center}
\begin{tabular}{cc}
  \includegraphics[width=5.5cm,angle=270]{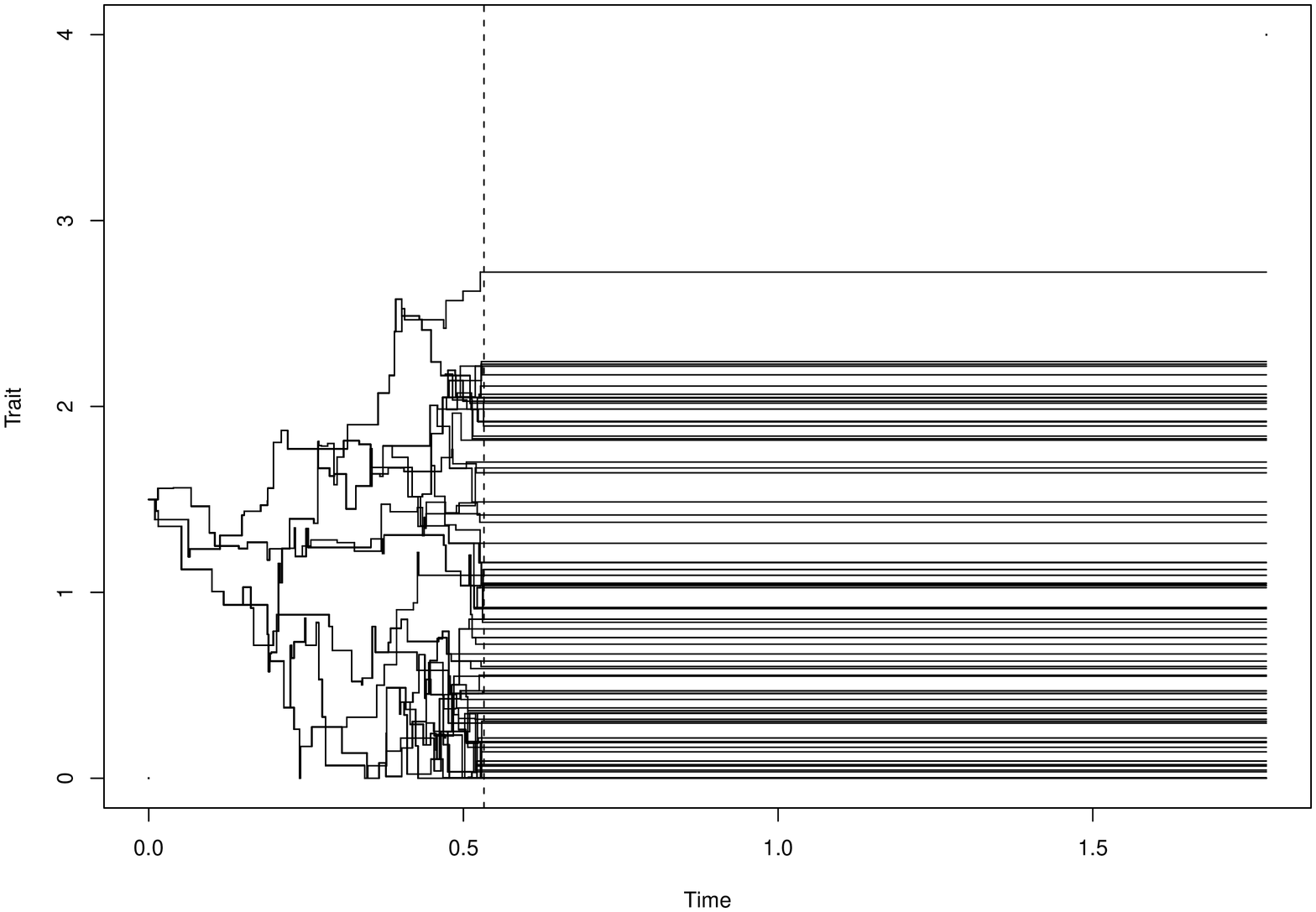} &   \includegraphics[width=5.5cm,angle=270]{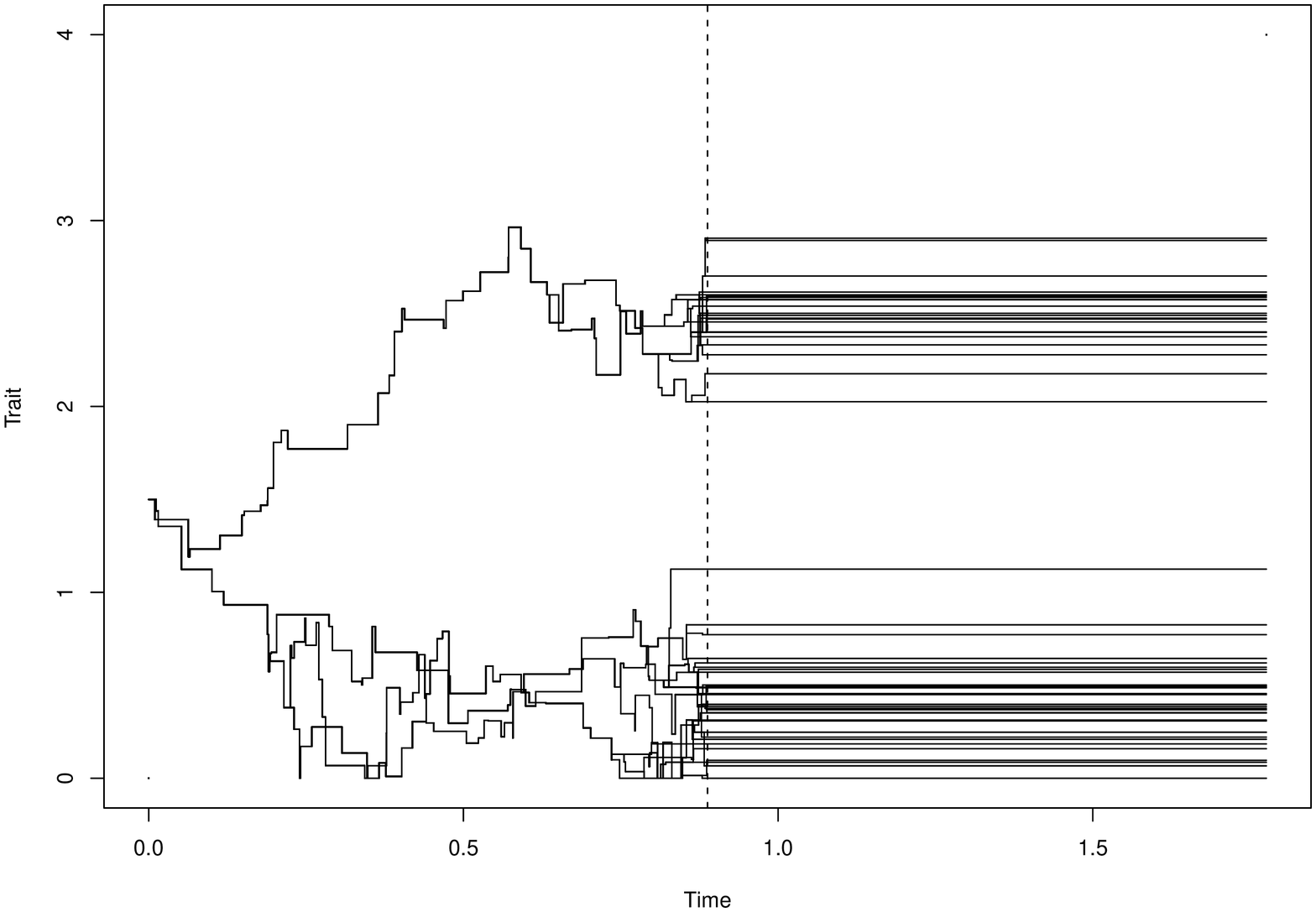}
     \end{tabular}
     \caption{{\small \textit{Simulation for Adler's fattened goats. $\alpha=10$, $\varepsilon=0.8$, $b=0.75$, $\sigma=1$, $p=1$, $K=50$ and $n=50$. The 100 initial particles are started at location $1.5$.}}}\label{figadler}
\end{center}
\end{figure}

\noindent Due to the form of the interaction,  the goats spread quickly in the whole space and separate into families with very old MRCAs, as observed in Fig. \ref{figadler} . The families become quickly disjoint and geographically isolated.

\section{The nonlinear  historical superprocess limit}

We now investigate a diffusive limit for the sequence of processes $X^n$ defined by (\ref{descriptionalgosuperprenormalisen}).
Let us firstly introduce a class of test functions which will be used to define the limiting process.
\begin{definition}\label{fonction-test:def}
For  real $\Co^2_{b}$-functions $g$ on $\R_+\times \R^d$ and  $G$ on $\R$ respectively, we define the continuous function $G_g$ on the path space $ \D_{\R^d}$ by
\begin{align}
\label{testfunction}
G_g(y) = G\Big(\int_{0}^T g(s,y_{s})ds\Big).
\end{align}
\end{definition}
\noindent Let us remark that the class generated by finite linear combinations of such functions is stable under addition and separates the points, as proven in Lemma \ref{lemme-passage1} in Appendix \ref{appendixA}.
Notice that if $y$  is a càdlàg path stopped at $t\in [0,T]$ then
\begin{align*}
G_g(y)=  G\Big(\int_{0}^t g(s,y_{s})ds + \int_t^T g(s,y_t)ds\Big).
\end{align*}Also, in the sequel, the following quantity will appear for $t\in [0,T]$ and $y\in \D_{\R^d}$:
\begin{multline}
\mathcal{D}^2G_g(t,y)= G'\Big(\int_0^T g(s,y_s)ds\Big)\int_t^T \Delta_x g(s,y_t)ds\\+ G''\Big(\int_0^T g(s,y_s)ds\Big)\sum_{i=1}^d \Big(\int_t^T \partial_{x_i} g(s,y_t)ds\Big)^2.\label{defD2}
\end{multline}This quantity generalizes the Laplacian operator to an infinite-dimensional setting. For instance, if $G(x)=x$ and if $g(s,y)=g(y)$ does not depend on time, we get
$\mathcal{D}^2G_g(t,y)=(T-t)\Delta g(y_t)$.

\bigskip \noindent
Note that   Dawson (\cite{dawson}, p. 203) and Etheridge (\cite{etheridgebook}, p. 24) introduce another class of test functions of the form
\begin{equation}\label{testfunction:dawson}
\varphi(y)=\prod_{j=1}^m g_j(y_{t_j}),\end{equation}for $m\in \N^*$, $0\leq t_1<\cdots <t_m$ and $\forall j\in \{ 1, \dots,m\},\, g_j\in \Co^{2}_b(\R^{d},\R)$.
This class is not convenient when dealing with càdlàg processes since the function $\varphi$ is  not continuous for the Skorohod topology. However,  these test functions will be used when $y$ is a continuous path. \\
\noindent If $y$  is a continuous path stopped at $t$ then $\varphi(y)=\prod_{j=1}^m g_j(y_{t_j\wedge t})$.
  For  a path $y\in \Co(\R_+,\R^d)$, a time $t>0$, we define
\begin{equation}
\label{laplacien}
\widetilde{\Delta}\varphi(t,y)=  \sum_{k=0}^{m-1}\ind_{[t_k,t_{k+1}[}(t)\Big(\prod_{j=1}^k g_j(y_{t_j}) \Delta\big(\prod_{j=k+1}^m g_j\big)(y_t)\Big),
\end{equation}
where $t_{0}=0$.

\noindent The following lemma links the test functions \eqref{testfunction} and \eqref{testfunction:dawson}. It is proved in Appendix \ref{appendixA}

\begin{lemme}\label{lien:fonctionstest}
Let $\varphi$ be a test function of the form \eqref{testfunction:dawson}. Then there exists test functions $(\varphi_q)$ of the form \eqref{testfunction} such that for every $y\in \Co(\R_+,\R^d)$ and $t\in [0,T]$, the sequences $(\varphi_q(y))_{q\in \N^*}$ and $(\mathcal{D}^2\varphi_{q}(t,y))_{q\in \N^*}$ are bounded uniformly in $q$, $t$ and $y$ and converge respectively to $\varphi(y)$ and $\widetilde{\Delta}\varphi(t,y)$.
\end{lemme}

\subsection{Main convergence result}

Let us assume that the initial conditions converge:
\begin{align}
\label{converge} \exists X_0\in \mathcal{M}_F(\mathbb{D}_{\R^d}),\, \lim_{n\rightarrow +\infty} X^n_0 = \bar X_0 \mbox{ for the weak convergence}.
\end{align}

\noindent The main theorem of this section states the convergence of the sequence $(X^n_t)_{n\in \N^*}$:
\begin{theorem}\label{propconvergence} Assume \eqref{converge} and \eqref{uniform}. Then the sequence $(X^n)_{n\in\N^*}$ converges in law in $\D(\R_+,\mathcal{M}_F(\D_{\R^d}))$
to the superprocess $\bar{X}\in \Co(\R_+,\mathcal{M}_F(\D_{\R^d}))$ characterized as follows, for test functions $G_g$ of the form  (\ref{testfunction}):\\
\begin{multline}
M^{G_g}_t=  \langle \bar{X}_t,G_g\rangle -\langle \bar{X}_0,G_g\rangle-\int_0^t \int_{\D_{\R^d}} \Big(
p\, r(s,y)\frac{\sigma^2}{2}\ \mathcal{D}^2G_g(s,y)\\
+\gamma(s,y,\bar{X}^s)G_g(y)\Big)\bar{X}_s(dy)\, ds\label{martingalesecondcas}
\end{multline}is a square integrable martingale with quadratic variation:
\begin{align}
\langle M^{G_g}\rangle_t= \int_0^t \int_{\D_{\R^d}}2\,r(s,y) G_g^2(y)\bar{X}_s(dy)\, ds,\label{crochetmartingalesecondcas}
\end{align}where $\mathcal{D}^2G_g(t,y)$ has been defined in \eqref{defD2} and
$\gamma(t,y,\bar{X}^t)$ defines the growth rate of individuals $y$ at time $t$ in the population $\bar{X}$:
\begin{equation}
\gamma(t,y,\bar{X}^t)=b(t,y)-d(t,y,\bar{X}^t).\label{def:g}
\end{equation}
\hfill $\Box$
\end{theorem}

\noindent For the proof of Theorem \ref{propconvergence}, we proceed in a compactness-uniqueness  manner. Firstly, we establish the tightness of the sequence
$(X^n)_{n\in \N^*}$ (Section \ref{sectiontension}) then use Prohorov's theorem and identify the limiting values as unique solution of the infinite-dimensional martingale problem \eqref{martingalesecondcas}, \eqref{crochetmartingalesecondcas}. The main difficulties in the proof are due to the interaction between the individual genealogies, implying the nonlinearity in the limit given by the term $d(t,y,\bar{X}^t)$.


\subsection{Tightness of $(X^n)_{n\in \N^*}$}\label{sectiontension}

In this subsection, we shall prove that:

\begin{prop}\label{proptightness}The sequence $(\mathcal{L}(X^n))_{n\in \N^*}$ is tight on $\mathcal{P}(\D(\R_+,\mathcal{M}_F(\D_{\R^d})))$.
\end{prop}

\noindent To do that, we use the following criterion characterizing the uniform tightness of measure-valued càdlàg processes by the compactness of the support of the measures and the uniform tightness of their masses. Moreover, a tricky coupling will allow us here to deal with the nonlinearity.

\begin{lemme}\label{theoremjakubowskihistorique}
The sequence of laws of $(X^n)_{n\in \N^*}$ is tight in ${\cal P}(\mathbb{D}(\R_+,\mathcal{M}_F(\D_{\R^d})))$ if \\
(i) $\forall T>0,\, \forall \varepsilon>0,\, \exists K\subset \D_{\R^d}\mbox{ compact},\,$
$$\sup_{n\in \N^*} \mathbb{P}\left(\exists t\in [0,T],\, X_t^n(K^c_T)>\varepsilon\right)\leq \varepsilon,$$
where $K^c_T$ is the complement set of
\begin{equation}
K_T=\left\{y^t,y^{t_-}\,\,\,|\,\,\,y\in K,\, t\in [0,T]\right\}\subset \D_{\R^d}.\label{def:K_T}
\end{equation}
(ii) $\forall \ G_g$ of the form  (\ref{testfunction}),  the family of laws $((\langle X^n,\,G_g\rangle))_{n\in \N^*}$ is uniformly tight in ${\cal P}(\mathbb{D}_{\R_+})$.
\end{lemme}

\begin{proof}[Sketch of proof]
Recall that $\mathbb{D}_{\R^d}$ embedded with the Skorohod topology is a Polish space. By Lemma 7.6 in Dawson and Perkins \cite{dawsonperkinsAMS}, the set $K_T$ is compact in $\mathbb{D}_{\R^d}$. Let $(K_{k,T})_{k\in \N^*}$ be the family of compact sets such that
$$\sup_{n\in \N^*} \mathbb{P}\left(\exists t\in [0,T],\, X_t^n(K^c_{k,T})>\varepsilon/2^k\right)\leq \varepsilon/2^k,$$
Then, the set
$$\mathfrak{K}=\bigcap_{k\in \N^*}\left\{\mu \in \mathcal{M}_F(\mathbb{D}_{\R^d})\,\,\,|\,\,\,\mu(K^c_{k,T})\leq \varepsilon/2^k\right\}$$is then relatively compact in $\mathcal{M}_F(\mathbb{D}_{\R^d})$ by the
Prohorov theorem, since it corresponds to a tight family of measures. We can then rewrite Point (i) of Lemma \ref{theoremjakubowskihistorique} in
$\forall T>0,\, \forall \varepsilon>0,\, \exists \mathfrak{K}\subset \mathcal{M}_F(\mathbb{D}_{\R^d})$ relatively compact,
$$\sup_{n\in \N^*}\mathbb{P}\left(\exists t\in [0,T],\, {X}^n_t\notin \mathfrak{K}\right)\leq \varepsilon.$$
 Moreover the class of functions $G_g$ separates the point and is closed under addition.  Thus Points (i) and (ii) of Lemma \ref{theoremjakubowskihistorique} allow us to apply the tightness result of Jakubowski  (Theorem 4.6 \cite{jakubowski}) and ensure that the sequence of the laws of  $({X}^n)_{n\in \N^*}$
is uniformly tight in ${\cal P}(\mathbb{D}(\R_+,\mathcal{M}_F(\mathbb{D}_{\R^d})))$.
\end{proof}

\begin{proof}[Proof of Proposition \ref{proptightness}]
We divide the proof into several steps.\\

\noindent \textbf{Step 1} Firstly, we prove Point (ii) of Lemma \ref{theoremjakubowskihistorique}. Let $T>0$ and $G_g$ be of the form (\ref{testfunction}).
\noindent  For every $t\in [0,T]$ and every $A>0$
\begin{align}
\mathbb{P}(|\langle X^n_t,G_g\rangle|>A)\leq \frac{\|G\|_\infty \sup_{n\in \N^*} \E(\langle X^n_t,1\rangle)}{A},
\end{align}which tends to 0 when $A$ tends to infinity thanks to \eqref{estimee_momentp}. This proves the uniform tightness of the family of marginal laws $\langle X^n_t,G_g\rangle$ for $n\in \N^*$ and $t$ fixed.
\noindent Then,  the  Aldous and Rebolledo criteria (\eg \cite{joffemetivier}) allow to prove the tightness of the process $\langle X^n,G_g\rangle$. For $\varepsilon>0$ and $\eta>0$, let us show that there exist $n_0\in \N^*$ and $\delta>0$ such that for all $n>n_0$ and  all stopping times $S_n<T_n<(S_n+\delta)\wedge T$,
\begin{align}
\mathbb{P}(|A_{T_n}^{n,G_g}-A_{S_n}^{n,G_g}|>\eta)\leq \varepsilon\quad \mbox{ and }\quad \mathbb{P}(|\langle M^{n,G_g}\rangle_{T_n}-\langle M^{n,G_g}\rangle_{S_n}|>\eta)\leq \varepsilon \label{aldous}
\end{align}where $A^{n,G_g}$ denotes the finite variation process in the r.h.s. of (\ref{PBM}).\\

\noindent Let us begin with some estimates. We fix $t\in [0,T]$, $h\in \R^d$ and a path $y\in \D_{\R^d}$ stopped at $t$. Using the Taylor-Lagrange formula, there exists $\theta_{y,t,h}\in (0,1)$ such that
\begin{multline}
G_g(y|t|y_{t_-}+h)-G_g(y)=  G\Big(\int_0^T g(s,(y|t|y_{t_-}+h)_s)ds\Big)-G\Big(\int_0^T g(s,y_s)ds\Big)\\
=  G'\Big(\int_0^T g(s,y_s)ds\Big)\Lambda(y,t,h)+  \frac{1}{2}G''\Big(\int_0^T g(s,y_s)ds+\theta_{y,t,h}\Lambda(y,t,h)\Big)   \Lambda(y,t,h)^2\label{etape34}
\end{multline}where
\begin{equation}
\Lambda(y,t,h)=\int_t^T \Big(g(s,y_{t_-}+h)-g(s,y_{t_-})\Big)ds
\end{equation}converges to zero when $h$ tends to zero. Another use of the Taylor-Lagrange formula  for the integrand in $\Lambda(y,t,h)$, shows  the existence a family $\eta_{y,t,h,s} \in (0,1)$ such that
\begin{equation}
\Lambda(y,t,h)=\int_t^T \Big(h\cdot\nabla_x g(s,y_{t_-})+ \frac{1}{2} \ ^th \ [ \mbox{Hess }g(s,y_{t_-}+\eta_{y,t,h,s} h) ]\  h\Big)ds.\label{etape35}
\end{equation}Using \eqref{etape34} and \eqref{etape35}, we integrate $G_g(y|t|y_{t_-}+h)-G_g(y)$ with respect to $K^n(dh)$. Since $k^n(h)$ is the density of the Gaussian distribution of mean 0 and covariance $\sigma^2 \Id /n$, integrals of odd powers and cross-products of the components of $h$ vanish. Thus:
\begin{multline*}
\int_{\R^d} \Big(G_g(y|t|y_{t_-}+h)-G_g(y)\Big) K^n(dh)=   G'\Big(\int_0^T g(s,y_s)ds\Big)\frac{\sigma^2 p}{2 n}  \int_t^T \Delta_x g(s,y_{t_-}+\eta_{t,y,h,s} h) ds\\
+    \frac{p\sigma^2}{2n}G''\Big(\int_0^T g(s,y_s)ds+\theta_{y,t,h}\Lambda(y,t,h)\Big)   \sum_{i=1}^d \Big(\int_0^T \partial_{x_i}g(s,y_{t_-})ds\Big)^2+\frac{C}{n^2}
\end{multline*}where $C$ is a constant depending on $G$, $g$, $\sigma^2$ and $p$ but not on $n$. Therefore
\begin{align}
\lim_{n\rightarrow +\infty} n \Big|\int_{\R^d} \Big(G_g(y|u|y_u+h)-G_g(y)\Big)K^n(dh)-\frac{\sigma^2 p}{2n}\mathcal{D}^2G_g(u,y)  \Big|=0.\label{limitD2}
\end{align}
Noting that for $G$ and $g$ in $\Co^2_b$, $\mathcal{D}^2G_g$ is bounded from the definition \eqref{testfunction}, we obtain the following upper bound:
\begin{multline}
\E(|A_{T_n}^{n,G_g}-A_{S_n}^{n,G_g}|)\leq  \delta\, \Big[\Big( \bar{R} \frac{p\sigma^2}{2} \big(\|\mathcal{D}^2G_g\|_\infty+1\big)+(\bar{B}+\bar{D})\|G\|_\infty\Big)\sup_{n\in \N^*}\E(\sup_{t\in [0,T]}\langle X^n_t,1\rangle)\\
+\|G\|_\infty \bar{U} \, \nu_d[0,T]\sup_{n\in \N^*}\E(\sup_{t\in [0,T]}\langle X^n_t,1\rangle^2)\Big].\label{etape20}
\end{multline}
For the quadratic variation process,
\begin{align}
 & \mathbb{E}(|\langle M^{n,G_g}\rangle_{T_n}-\langle M^{n,G_g}\rangle_{S_n}|)\nonumber\\
 \leq &  \|G\|_\infty^2 \delta \;\Big[\Big(2\bar{r} +\frac{\bar{b}+\bar{d}}{n}\Big) \sup_{n\in \N^*}
 \E(\sup_{t\in [0,T]}\langle X^n_t,1\rangle)+\frac{\bar{U}\ \nu_d[0,T]}{n}\sup_{n\in \N^*}\E(\sup_{t\in [0,T]}\langle X^n_t,1\rangle^2)\Big].\label{etape21}
\end{align}We thus obtain (\ref{aldous}) by applying the Markov inequality and using the moment estimates of Proposition \ref{prop_usefulresults}.\\

\noindent \textbf{Step 2} Let us now check that Point (i) of Lemma \ref{theoremjakubowskihistorique} is satisfied. We follow here ideas
introduced by Dawson and Perkins \cite{dawsonperkinsAMS}
who proved the tightness of a system of independent  historical branching Brownian particles. Here, we have interacting particles, which makes the proof much harder.

\medskip \noindent  Let $T\in \R_+$ and $\varepsilon>0$ and  $K$ be a compact set of $\D_{\R^d}$. We denote by $K^t=\{y^t\, |\, y\in K\}\subset \D_{\R^d}$ the set of the paths of $K$  stopped at time $t$. Recall that $K_T$ defined in \eqref{def:K_T} is the set of the paths of $K$ and of their left-limited paths stopped at any time before time $T$. Let us define the stopping time
\begin{equation}
S^n_\varepsilon=\inf\{t\in \R_+\, |\, X^n_t(K^c_T)>\varepsilon\}.
\end{equation}From this definition,
\begin{equation}
\P\big(\exists t\in [0,T],\, X^n_t(K_T^c)>\varepsilon\big)=\P(S^n_\varepsilon<T).\label{etape24}
\end{equation}Our purpose is to prove that it is possible to choose $K$ and $n_0$ such that $\sup_{n\geq n_0}\P(S^n_\varepsilon<T)\leq \varepsilon$.
We decompose \eqref{etape24} by considering the more tractable $X^n_T((K^T)^c)$ and write
\begin{align}
\P(S^n_\varepsilon<T)= & \P\Big(S^n_\varepsilon<T,\ X^n_T((K^T)^c)>\frac{\varepsilon}{2}\Big)+\P\Big(S^n_\varepsilon<T,\ X^n_T((K^T)^c)\leq \frac{\varepsilon}{2}\Big)\nonumber\\
\leq & \frac{2}{\varepsilon}\E\big(X^n_T((K^T)^c)\big)+\P\Big(S^n_\varepsilon<T,\ X^n_T((K^T)^c)\leq \frac{\varepsilon}{2}\Big)\label{etape28}
\end{align}by using the Markov inequality. We will show in Steps 3 to 5 that there exists $\eta\in (0,1)$ such that for $n$ large enough,
\begin{equation}
\P\Big(S^n_\varepsilon<T,\ X^n_T((K^T)^c)\leq \frac{\varepsilon}{2}\Big)\leq \P(S^n_\varepsilon<T)(1-\eta).\label{etape27}
\end{equation}Together with \eqref{etape28}, this entails that
\begin{align}
\P(S^n_\varepsilon<T)\leq \frac{2\E\big(X^n_T((K^T)^c)\big)}{\varepsilon \eta}.\label{etape12}
\end{align}In Step 6, we will also prove that the compact set $K$ can be chosen such that
\begin{equation}
\E\big(X^n_T((K^T)^c)\big)<\frac{\varepsilon^2\eta}{2}.\label{reste_a_prouver3}
\end{equation}This will conclude the proof.\\

\noindent \textbf{Step 3} Let us prove \eqref{etape27}. Heuristically, the event $\ \{S^n_\varepsilon<T,\ X^n_T((K^T)^c)\leq \frac{\varepsilon}{2}\}$ means that more than half the trajectories that exited $K$ before $S^n_\varepsilon$ have died at time $T$. On the set $\{S^n_\varepsilon<T\}$, $y^{S^n_\varepsilon}\notin K^{S^n_\varepsilon}$ implies $y^T\notin K^T$. Indeed, if a path stopped at time
$S^n_\varepsilon$ does not belong to $K$, then it is also true when it is stopped at $T>S^n_\varepsilon$. Thus on $\{S^n_\varepsilon<T\}$:
\begin{equation}
X^n_T((K^T)^c)\geq X^n_T(\{y^{S^n_\varepsilon}\notin K^{S^n_\varepsilon}\}).\label{etape10}
\end{equation}
\unitlength=0.6cm
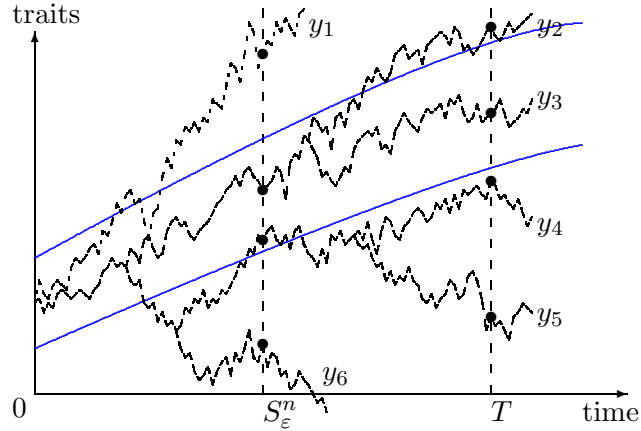
\begin{figure}[ht]
  \begin{center}
    \begin{picture}(15,10)
      \put(1,1){\vector(1,0){13}} \put(1,1){\vector(0,1){8}}
      \put(0.5,0.5){0}
      \put(0.5,9.2){traits}
      \put(13,0.4){time}
      \dashline{0.2}(6,1)(6,9.5)
      \put(6,0.4){$S^n_\varepsilon$}
      \dashline{0.2}(11,1)(11,9.5)
      \put(11,0.4){$T$}
      {\color{blue}\qbezier(1,2)(10,6)(13,6.5)
      \qbezier(1,4)(10,9)(13,9.2)
      }
      \dashline{0.1}
      (1,3)(1.1,3.5)(1.2,3.2)(1.3,3.1)(1.4,3.4)(1.5,3.5)(1.6,3.4)(1.7,3.3)(1.8,3.8)(1.9,4)
      (2,3.9)(2.1,4)(2.2,4.2)(2.3,4.3)(2.4,4.1)(2.5,4.6)(2.6,4.5)(2.7,4.8)(2.8,5.2)(2.9,5)
      (3,5.4)(3.1,5.5)(3.2,5.2)(3.3,5.3)(3.4,4.8)(3.5,4.4)(3.6,5)(3.7,5.1)(3.8,5.8)(3.9,6.2)
      (4,6)(4.1,6.6)(4.2,6.5)(4.3,6.9)(4.4,6.8)(4.5,7)(4.6,6.8)(4.7,7.1)(4.8,7)(4.9,7.5)
      (5,7.4)(5.1,7.3)(5.2,7.8)(5.3,8.3)(5.4,8.2)(5.5,7.9)(5.6,8.5)(5.7,9.2)(5.8,9)(5.9,8.2)
      (6,8.5)(6.1,8.8)(6.2,8.7)(6.3,9.2)(6.4,9)(6.5,9.3)(6.6,9.1)(6.7,8.9)(6.8,9.4)(6.9,9.5)
      \put(6,8.5){\circle*{0.25}}
      \put(7,9){$y_1$}
      \dashline{2}
      (1,3)(1.1,3.1)(1.2,2.8)(1.3,3)(1.4,2.9)(1.5,3.2)(1.6,3.3)(1.7,3.2)(1.8,3.1)(1.9,3.2)
      (2,3.3)(2.1,3.2)(2.2,3.1)(2.3,3.3)(2.4,3.4)(2.5,3.7)(2.6,3.8)(2.7,4)(2.8,3.8)(2.9,3.9)
      (3,3.8)(3.1,4.2)(3.2,4.1)(3.3,3.9)(3.4,4.4)(3.5,4.5)(3.6,4.2)(3.7,4)(3.8,4.1)(3.9,4.6)
      (4,4.8)(4.1,4.9)(4.2,4.7)(4.3,4.8)(4.4,4.5)(4.5,4.4)(4.6,4.6)(4.7,4.7)(4.8,4.8)(4.9,5)
      (5,5.2)(5.1,5.3)(5.2,5.5)(5.3,5.6)(5.4,5.8)(5.5,5.6)(5.6,5.9)(5.7,6)(5.8,5.8)(5.9,5.7)
      (6,5.5)(6.1,5.6)(6.2,5.7)(6.3,6)(6.4,5.9)(6.5,5.3)(6.6,6.1)(6.7,6.3)(6.8,6.2)(6.9,6.7)
      (7,6.6)(7.1,7.1)(7.2,7)(7.3,6.8)(7.4,6.9)(7.5,6.4)(7.6,7.1)(7.7,7.2)(7.8,6.9)(7.9,7.4)
      (8,7.5)(8.1,7.3)(8.2,7.7)(8.3,8)(8.4,7.9)(8.5,8.2)(8.6,8)(8.7,7.9)(8.8,8.1)(8.9,8.5)
      (9,8.6)(9.1,8.7)(9.2,8.4)(9.3,8.5)(9.4,8.6)(9.5,8.3)(9.6,8.7)(9.7,8.8)(9.8,8.9)(9.9,9)
      (10,8.9)(10.1,8.7)(10.2,9.1)(10.3,9)(10.4,8.6)(10.5,8.7)(10.6,8.5)(10.7,8.6)(10.8,8.8)(10.9,9)
      (11,9.1)(11.1,8.9)(11.2,8.8)(11.3,9)(11.4,9.2)(11.5,9)(11.6,9.1)(11.7,9.2)(11.8,9.3)(11.9,9.4)
      \put(6,5.5){\circle*{0.25}}
      \put(11,9.1){\circle*{0.25}}
      \put(12,9){$y_2$}
      \dashline{2}
      (7,6.6)(7.1,6.4)(7.2,6.2)(7.3,5.8)(7.4,5.9)(7.5,5.7)(7.6,5.6)(7.7,5.8)(7.8,6)(7.9,5.9)
      (8,6.4)(8.1,6.6)(8.2,6.7)(8.3,6.8)(8.4,6.7)(8.5,6.9)(8.6,6.4)(8.7,6.5)(8.8,6.3)(8.9,6.8)
      (9,7)(9.1,7.2)(9.2,7)(9.3,6.9)(9.4,6.8)(9.5,7.1)(9.6,7.3)(9.7,7.2)(9.8,7.3)(9.9,7.4)
      (10,7.1)(10.1,7.5)(10.2,7.4)(10.3,7.6)(10.4,7.2)(10.5,7)(10.6,7.1)(10.7,7.4)(10.8,7)(10.9,7.1)
      (11,7.2)(11.1,7.3)(11.2,7.5)(11.3,7.2)(11.4,6.9)(11.5,7)(11.6,7.1)(11.7,7.3)(11.8,7.2)(11.9,7.5)
      \put(11,7.2){\circle*{0.25}}
      \put(12,7.4){$y_3$}
      \dashline{2}
      (3,3.8)(3.1,3.5)(3.2,3.4)(3.3,3.5)(3.4,3)(3.5,2.9)(3.6,3)(3.7,2.6)(3.8,2.8)(3.9,2.5)
      (4,2.6)(4.1,2.4)(4.2,2.5)(4.3,2.9)(4.4,2.8)(4.5,3.1)(4.6,3)(4.7,3.3)(4.8,2.9)(4.9,3.1)
      (5,3)(5.1,3.3)(5.2,3.2)(5.3,3.4)(5.4,3.6)(5.5,3.9)(5.6,3.8)(5.7,4.2)(5.8,4.1)(5.9,4.5)
      (6,4.4)(6.1,4.6)(6.2,4.7)(6.3,4.4)(6.4,4.3)(6.5,4.5)(6.6,4.2)(6.7,4.6)(6.8,4.7)(6.9,4.3)
      (7,4.1)(7.1,4)(7.2,3.8)(7.3,4.2)(7.4,3.9)(7.5,4.5)(7.6,4.4)(7.7,4.6)(7.8,4.3)(7.9,4.4)
      (8,4.5)(8.1,4.6)(8.2,4.4)(8.3,4.8)(8.4,4.4)(8.5,4.1)(8.6,4.6)(8.7,4.9)(8.8,4.8)(8.9,4.5)
      (9,4.4)(9.1,4.7)(9.2,4.8)(9.3,5)(9.4,5.1)(9.5,5)(9.6,4.8)(9.7,4.9)(9.8,5.2)(9.9,5.1)
      (10,5.3)(10.1,5)(10.2,5.4)(10.3,5.1)(10.4,5.5)(10.5,5.4)(10.6,5.5)(10.7,5.6)(10.8,5.4)(10.9,5.3)
      (11,5.7)(11.1,5.6)(11.2,5.3)(11.3,5.4)(11.4,5.2)(11.5,5.1)(11.6,5.3)(11.7,5)(11.8,4.7)(11.9,4.9)
      \put(6,4.4){\circle*{0.25}}
      \put(11,5.7){\circle*{0.25}}
      \put(12,4.6){$y_4$}
      \dashline{2}
      (8,4.5)(8.1,4.3)(8.2,4.2)(8.3,4.3)(8.4,4)(8.5,3.8)(8.6,3.9)(8.7,3.5)(8.8,3.7)(8.9,3.4)
      (9,3.6)(9.1,3.5)(9.2,3.8)(9.3,3.5)(9.4,3.4)(9.5,3.1)(9.6,3)(9.7,3.2)(9.8,2.9)(9.9,3.1)
      (10,3.2)(10.1,3.5)(10.2,3.4)(10.3,3.1)(10.4,3.2)(10.5,3.6)(10.6,3.7)(10.7,3.5)(10.8,3.1)(10.9,2.4)
      (11,2.7)(11.1,2.5)(11.2,2.6)(11.3,2.4)(11.4,2.2)(11.5,2.7)(11.6,2.8)(11.7,3)(11.8,2.9)(11.9,2.8)
      \put(11,2.7){\circle*{0.25}}
      \put(12,2.6){$y_5$}
      \dashline{2}
      (4.1,2.4)(4.2,2)(4.3,1.8)(4.4,1.9)(4.5,1.5)(4.6,1.3)(4.7,1.4)(4.8,1.6)(4.9,1.2)
      (5,1.5)(5.1,1.7)(5.2,1.8)(5.3,1.4)(5.4,1.9)(5.5,1.7)(5.6,2)(5.7,2.4)(5.8,1.8)(5.9,1.7)
      (6,2.1)(6.1,1.6)(6.2,1.9)(6.3,1.8)(6.4,1.4)(6.5,1.5)(6.6,1.1)(6.7,1.2)(6.8,1.5)(6.9,1.3)
      (7,1.2)(7.1,1)(7.2,0.8)(7.3,0.9)(7.4,0.6)
      \put(6,2.1){\circle*{0.25}}
      \put(7.3,1.3){$y_6$}
\end{picture}
  \end{center}
  \vspace{-0.5cm}
  \caption{{\small\textit{The compact $K$ is the region between the two lines. The paths which are drawn correspond to the support  of $X^n_T$. $X^n_T((K^T)^c)$ counts the trajectories that do not belong to $K$ between $0$ and $T$: here we have $y_1$, $y_2$, $y_4$, $y_5$ and $y_6$. The quantity $X^n_T(\{y^{S^n_\varepsilon}\notin K^{S^n_\varepsilon}\})$ counts the trajectories, at time $T$, that don't belong to $K$ between $0$ and $S^n_\varepsilon$. Here, we have $y_1$, $y_4$, $y_5$ and $y_6$ ; although $y_2$ does not belong to $K$ between time $0$ and $T$, it belongs to $K$ between $0$ and $S^n_\varepsilon$. To obtain the trajectories accounting for $X^n_T(\{y^{S^n_\varepsilon}\notin K^{S^n_\varepsilon}\})$, we count the descendants of the 3 points at time $S^n_\varepsilon$ corresponding to trajectories $y^{S^n_\varepsilon}\notin K^{S^n_\varepsilon}$. We can also check that relation \eqref{etape10} is satisfied. }}}\label{fig2}
\end{figure}

\bigskip \noindent
Hence:
\begin{align}
\P\Big(S^n_\varepsilon< T,\, X^n_T((K^T)^c)\leq \frac{\varepsilon}{2}\Big)\leq &
\P\Big(S^n_\varepsilon< T,\, X^n_T(\{y^{S^n_\varepsilon}\notin K^{S^n_\varepsilon}\})\leq \frac{\varepsilon}{2}\Big)\nonumber\\
= & \E\Big(\ind_{S^n_\varepsilon< T}\P\big(X^n_{S^n_\varepsilon+(T-S^n_\varepsilon)}(\{y^{S^n_\varepsilon}\notin K^{S^n_\varepsilon}\})
\leq \varepsilon/2\, |\, \mathcal{F}_{S^n_\varepsilon}\big)\Big).\label{etape11}
\end{align}
Our purpose is to obtain an upper bound of the form $(1-\eta)$ with $\eta\in (0,1)$ for the probability under the expectation
in the r.h.s. of (\ref{etape11}).
This term is the probability that the population issued from particles which at time $S^n_\varepsilon$ satisfy $y^{S^n_\varepsilon}\notin K^{S^n_\varepsilon}$, has a size smaller than $\varepsilon/2$. In view of \eqref{etape11}, we will work on the set $\{S^n_\varepsilon<T\}$ until the end of the proof.\\
Using \eqref{estimee_momentp} and Markov's inequality, it is possible, for any $\eta>0$, to choose $N>0$  large enough such that
\begin{equation}\P\Big(\sup_{S^n_\varepsilon \leq s\leq T}\langle X^n_s,1\rangle>N\, |\, \mathcal{F}_{S^n_\varepsilon}\Big)<\eta.\label{etape13}\end{equation}
We need to introduce some coupling with independent trajectories. Let us define the process $(Z^{n}_t(dy))_{t\in \R_+}$ as follows. The independent particles of $Z^{n}$ are started at time $S^n_\varepsilon$ with the trajectories of $X^n_{S^n_\varepsilon}$ such that
 $\{y^{S^n_\varepsilon}\notin K^{S^n_\varepsilon}\}$ and  the initial condition is
\begin{equation}
Z_{S^n_\varepsilon}^{n}(dy)=\ind_{y^{S^n_\varepsilon}\notin K^{S^n_\varepsilon}}X^n_{S^n_\varepsilon}(dy).
\end{equation}
Their birth and death rates are $n r(t,y)$ and $n r(t,y)+\bar{D}+\bar{U}N$. By a coupling argument, we have
\begin{multline}
\P\big(X^n_{T}(\{y^{S^n_\varepsilon}\notin K^{S^n_\varepsilon}\})
\leq \varepsilon/2\, |\, \mathcal{F}_{S^n_\varepsilon}\Big)  \label{etape18}\\
\begin{aligned}
\leq  & \P\Big(\langle Z^{n}_{T},1\rangle\leq \frac{\varepsilon}{2} ; \sup_{S^n_\varepsilon \leq s\leq T} \langle X^n_{s},1\rangle \leq N\, |\, \mathcal{F}_{S^n_\varepsilon}\Big)
+ \P\Big(\sup_{S^n_\varepsilon \leq  s\leq T}\langle X^n_{s},1\rangle >N\, |\, \mathcal{F}_{S^n_\varepsilon}\Big)\\
\leq & 1-\P\Big(\inf_{s\in [S^n_\varepsilon,T]}\langle Z^n_{s},1\rangle >\frac{\varepsilon}{2}\, |\, \mathcal{F}_{S^n_\varepsilon}\Big) + \eta.
\end{aligned}
\end{multline}
If we can exhibit $\eta>0$ such that for $n$ large enough
\begin{equation}
\E\Big(\ind_{S^n_\varepsilon< T}\ \P\big(\inf_{s\in [S^n_\varepsilon,T]}\langle Z^n_{s},1\rangle >\frac{\varepsilon}{2}\, |\, \mathcal{F}_{S^n_\varepsilon}\big)\Big)>2\eta \P(S^n_\varepsilon<T),\label{reste_a_prouver}
\end{equation}then from \eqref{etape11} and (\ref{etape18}), the r.h.s. of \eqref{etape11} is strictly smaller
than $(1-2\eta+\eta)\P(S^n_\varepsilon<T)=(1-\eta)\P(S^n_\varepsilon<T)$, and \eqref{etape27} will be proved.\\

\noindent \textbf{Step 4} Let us prove (\ref{reste_a_prouver}). Notice that on the set $\{S^n_\varepsilon<T\}$,
\begin{align*}
\langle Z^n_{S^n_\varepsilon},1\rangle=X^n_{S^n_\varepsilon}(\{y^{S^n_\varepsilon}\notin K^{S^n_\varepsilon}\})
\geq  X^n_{S^n_\varepsilon}(\{y^{S^n_\varepsilon}\notin K_{S^n_\varepsilon}\})
\geq  X^n_{S^n_\varepsilon}(\{y^{S^n_\varepsilon}\notin K_{T}\})=X^n_{S^n_{\varepsilon}}(K_T^c)>\varepsilon.
\end{align*}
By coupling arguments (using deletions of particles in the initial condition $Z^n_{S^n_\varepsilon}$), and since we are considering a lower bound  with an infimum in \eqref{reste_a_prouver}, we can consider without restriction that $\langle Z^n_{S^n_\varepsilon},1\rangle=([n\varepsilon]+1)/n,$ where $[x]$ denotes the integer part of $x$.\\
Let us establish a diffusion approximation of $\langle Z^n_{S^n_\varepsilon+.},1\rangle$ when $n$ is large. We know that for any $t\geq 0$, the process
\begin{equation}
\langle Z^n_{S^n_\varepsilon+t},1\rangle=\langle Z^n_{S^n_\varepsilon},1\rangle-\big(\bar{D}+\bar{U}N\big)\int_{0}^t \langle Z^n_{S^n_\varepsilon+s},1\rangle ds+M^{n,Z}_{t}\label{etape22}
\end{equation}where $M^{n,Z}$ is a square integrable martingale such that for all $ s\leq t$,
\begin{multline}
 2\underline{R}\int_s^t \langle Z^n_{S^n_\varepsilon+u},1\rangle du\leq \langle M^{n,Z}\rangle_t-\langle M^{n,Z}\rangle_s =\int_s^t \Big\langle Z^n_{S^n_\varepsilon+u},2 r(S^n_\varepsilon+u,.)+\frac{\bar{D}+\bar{U}N}{n}\Big\rangle du\\
  \leq \big(2\bar{R}+\frac{\bar{D}+\bar{U}N}{n}\big)\int_s^t \langle Z^n_{S^n_\varepsilon+u},1\rangle du.\label{etape23}
\end{multline}
Proposition  \ref{prop_usefulresults} and adaptations of (\ref{etape20}) and (\ref{etape21}) allow us to establish that  the  laws of $(\langle Z^n_{S^n_\varepsilon+.},1\rangle, \langle M^{n,Z}\rangle_.)$ are uniformly tight in ${\cal P}(\D(\R_+,\R_+^2))$. As a consequence, there exists a subsequence, denoted again by $(\langle Z^{n}_{S^n_\varepsilon+.},1\rangle, \langle M^{n,Z}\rangle_.)_{n\in \N^*}$ for simplicity, that converges in distribution to a limit,
say $(\mathcal{Z}, \mathcal{A})$ where $\mathcal{Z}$ and $\mathcal{A}$ are necessarily continuous. Let us define on the canonical space
\begin{equation}
\mathcal{N}_t = \mathcal{Z}_t- \varepsilon+\int_0^t (\bar{D}+\bar{U}N) \mathcal{Z}_s ds\label{def_mathcalN},
\end{equation}
and prove that it is a martingale.
Let $0\leq s_1\leq \cdots \leq s_k<s<t$ and let $\phi_1,\cdots \phi_k$ be bounded measurable functions on $\R_+$. We define
\begin{equation}
\label{psi}
\Psi(\mathcal{Z})=\phi_1(\mathcal{Z}_{s_1})\cdots \phi_k(\mathcal{Z}_{s_k})\Big\{\mathcal{Z}_t-\mathcal{Z}_s+\int_s^t (\bar{D}+\bar{U}N)\mathcal{Z}_u\, du\Big\}.
\end{equation}
From (\ref{etape22}), $\E(\Psi(\langle Z^n_{S^n_\varepsilon+.},1\rangle))=0$. Similarly to Prop. \ref{prop_usefulresults} (ii), we can prove from the SDE \eqref{etape22} that $\E(\sup_{t\in [0,T]}\langle Z^n_{S^n_\varepsilon+t},1\rangle^3)<+\infty$ for any $T>0$. Then the sequence $(\Psi(\langle Z^n_{S^n_\varepsilon+.},1\rangle))_{n\in \N^*}$
is uniformly integrable and by the continuity of $\Psi$, $\lim_{n\rightarrow +\infty}\E\big(\Psi(\langle Z^n_{S^n_\varepsilon+.},1\rangle\big)
=\E(\Psi(\mathcal{Z})).$ Then we deduce that  $\E(\Psi(\mathcal{Z}))=0$. Hence, $\mathcal{N}$ is a continuous square integrable martingale, and  Theorem 6.1 p. 341 in
Jacod and Shiryaev \cite{jacod} implies that  its quadratic variation process
is $\langle \mathcal{N}\rangle= \mathcal{A}$.\\
Moreover, using the Skorokhod representation theorem (see \eg \cite{billingsley_probability_and_measure} Th. 25.6 p.333), there exist a random sequence
$(\widetilde{\mathcal{Z}}^n,\widetilde{\mathcal{A}}^n)_{n\in \N^*}$ and a random couple
$(\widetilde{\mathcal{Z}},\widetilde{\mathcal{A}})$ defined on the same probability space, distributed respectively as
$(\langle Z^{n}_{S^n_\varepsilon+.},1\rangle,\langle M^{n,Z}\rangle_.)_{n\in \N^*}$ and $(\mathcal{Z}, \mathcal{A})$, and such that
\begin{equation}
\lim_{n\rightarrow +\infty}(\widetilde{\mathcal{Z}}^n,\widetilde{\mathcal{A}}^n)
=(\widetilde{\mathcal{Z}},\widetilde{\mathcal{A}})\qquad \mbox{ a.s.}.
\end{equation}
Then, from (\ref{etape23}), we have a.s. that for all $0\leq s\leq t$,
\begin{equation}
2 \underline{R}\int_s^t  \widetilde{\mathcal{Z}}_u\, du\leq \widetilde{\mathcal{A}}_t-\widetilde{\mathcal{A}}_s
\leq 2 \bar{R}\int_s^t \widetilde{\mathcal{Z}}_u\, du.\label{encadrement_A}
\end{equation}
This implies (see e.g. Rudin \cite[Chapter 8]{rudin}) that $\widetilde{\mathcal{A}}$ is a.s. an absolutely continuous function and that there exists a random $\mathcal{F}_t$-measurable function $\rho(u)$ such that $\forall u\in \R_+,\,
2\underline{R}\leq \rho(u)\leq 2\bar{R}$ and
\begin{equation}
\widetilde{\mathcal{A}}_t=\int_0^t \rho(u) \widetilde{\mathcal{Z}}_u\, du\qquad \mbox{ a.s.}.
\end{equation}
Therefore, there exists a standard Brownian motion $(B_t)_{t\in \R_+}$ such that almost surely:
\begin{equation}
\widetilde{\mathcal{N}}_t = \widetilde{\mathcal{Z}}_t- \varepsilon+\int_0^t (\bar{D}+\bar{U}N) \widetilde{\mathcal{Z}}_s ds=\int_0^t \sqrt{\rho(u)\widetilde{\mathcal{Z}}_u} dB_u.\label{def_mathcalN}
\end{equation}

\bigskip
Now that the diffusive limit for $\langle Z^n_{S^n_\varepsilon+.},1\rangle$ has been obtained, let us return to \eqref{reste_a_prouver}:
\begin{align}
\P\Big(\inf_{s\in [S^n_\varepsilon,T]}\langle Z^n_{s},1\rangle >\frac{\varepsilon}{2}\, |\, \mathcal{F}_{S^n_\varepsilon}\Big) & \ind_{S^n_\varepsilon< T}
= \P\Big(\inf_{u\in [0,T-S^n_\varepsilon]}\langle Z^n_{S^n_\varepsilon+u},1\rangle >\frac{\varepsilon}{2}\, |\, \mathcal{F}_{S^n_\varepsilon}\Big)\ind_{S^n_\varepsilon< T}\nonumber\\
\geq & \ \P\Big(\inf_{u\in [0,T]} (y|s|\widetilde{\mathcal{Z}}^n)_{s+u} >\frac{\varepsilon}{2}\Big)\Big|_{y=\langle Z^n_{.\wedge S^n_\varepsilon},1\rangle,\ s=S^n_\varepsilon}\ind_{S^n_\varepsilon< T}.\label{etape19}
\end{align}Notice that for all $y$ and $s$,
\begin{align}\lim_{n\rightarrow +\infty} \P\Big(\inf_{u\in [0,T]} (y|s|\widetilde{\mathcal{Z}}^n)_{s+u} > \frac{\varepsilon}{2}\Big)= & \P\Big(\inf_{u\in [0,T]} (y|s|\widetilde{\mathcal{Z}})_{s+u} \geq \frac{\varepsilon}{2}\Big)=  \P_{y,s}\Big(\inf_{u\in [0,T]} \widetilde{\mathcal{Z}}_u \geq \frac{\varepsilon}{2}\Big),\label{etape30}
\end{align}
where the notation $\P_{y,s}$ reminds that the distribution of $\widetilde{\mathcal{Z}}$ depends on $\rho$ which may itself depend on $(y,s)$.
We need some uniformity of the convergence in \eqref{etape30}, with respect to $y$ and $s$.\\
For $\zeta>0$ and $(z,r)\in\D\times [0,T]$, we denote by $B((z,r),\zeta)$ the open ball centered at $(z,r)$ with radius $\zeta$. There exists $\zeta>0$ small enough such that for all $(y,s)\in \mathcal{B}((z,r),\zeta)$ and $\mathcal{Z}\in \D$,  $\dSk\big((y|s|\mathcal{Z}),(z|r|\mathcal{Z})\big)<\varepsilon/4$ where $\dSk$ is the Skorokhod distance on $\D$ (see Proposition \ref{prop_Sk} in appendix). As a consequence, for this choice of $\zeta$ and all $n\in \N^*$,
\begin{equation}
\P\Big(\inf_{u\in [0,T]}(z|r|\widetilde{\mathcal{Z}}^n)_{r+u}>\frac{3\varepsilon}{4}\Big)\leq \P\Big(\inf_{u\in [0,T]}(y|s|\widetilde{\mathcal{Z}}^n)_{s+u}>\frac{\varepsilon}{2}\Big).\label{etape32}
\end{equation}
Let $\xi>0$ be a small positive number. Since the sequence of laws of $\langle Z^n_.,1\rangle$ is uniformly tight, there exists a compact set $K_\xi$ of $\D_{\R^d}$ such that for $n$ large enough, $\P\big(\langle Z^n_.,1\rangle \notin K_\xi\big)<\xi$. Since $K_\xi\times [0,T]$ is compact, there exists a finite sequence $(z_i,r_i)_{1\leq i\leq M}$ with $M=M(\xi)\in \N^*$ such that $$K_\xi\times [0,T]\subset \bigcup_{i=1}^{M(\xi)} B\big((z_i,r_i),\zeta\big).$$
With an argument similar to \eqref{etape30}, there exists for each $i\in \{1,\dots,M\}$, an integer $n_i$ such that for all $n\geq n_i$,
\begin{align}
\P\Big(\inf_{u\in [0,T]}(z_i|r_i|\widetilde{\mathcal{Z}}^n)_{r_i+u}>\frac{3\varepsilon}{4}\Big)>\frac{1}{2}\P_{z_i,r_i}\Big(\inf_{u\in [0,T]}\widetilde{\mathcal{Z}}_u\geq \frac{3\varepsilon}{4}\Big).
\end{align}Hence, thanks to \eqref{etape32}, we obtain that for all $(y,s)\in K_\xi\times [0,T]$ and $n\geq \max_{1\leq i\leq M(\xi)} n_i$,
\begin{multline}
 \P\Big(\inf_{u\in [0,T]} (y|s|\widetilde{\mathcal{Z}}^n)_{s+u} \geq \frac{\varepsilon}{2}\Big)\geq \min_{1\leq i\leq M}\P\Big(\inf_{u\in [0,T]}(z_i|r_i|\widetilde{\mathcal{Z}}^n)_{r_i+u}>\frac{3\varepsilon}{4}\Big)\\
>\min_{1\leq i\leq M(\xi)}\frac{1}{2}\P_{z_i,r_i}\Big(\inf_{u\in [0,T]}\widetilde{\mathcal{Z}}_u\geq \frac{3\varepsilon}{4}\Big)=:2\eta(\xi).\label{reste_a_prouver2}
\end{multline}
Then, from \eqref{etape19} and \eqref{reste_a_prouver2}, the left hand side of \eqref{reste_a_prouver} has the lower bound:
\begin{multline}
\E\Big(\ind_{S^n_\varepsilon< T}\ \P\big(\inf_{s\in [S^n_\varepsilon,T]}\langle Z^n_{s},1\rangle >\frac{\varepsilon}{2}\, |\, \mathcal{F}_{S^n_\varepsilon}\big)\Big)\label{etape2}\\
\begin{aligned}
\geq & \E\Big(\P\Big(\inf_{u\in [0,T]} (y|s|\widetilde{\mathcal{Z}}^n)_{s+u} >\frac{\varepsilon}{2}\Big)\Big|_{y=\langle Z^n_{.\wedge S^n_\varepsilon},1\rangle,\ s=S^n_\varepsilon}\ind_{S^n_\varepsilon< T}\ind_{\langle Z^n_.,1\rangle \in K_\xi}\Big)\\
\geq & \ 2\eta(\xi)\ \P\big(S^n_\varepsilon< T,\ \langle Z^n_.,1\rangle \in K_\xi\big).
\end{aligned}
\end{multline}The term $\P\big(S^n_\varepsilon< T,\ \langle Z^n_.,1\rangle \in K_\xi\big)$ in the right hand side converges to $\P(S^n_\varepsilon<T)$ when $\xi$ tends to zero, and there exists $\xi_0>0$  small enough such that this term is larger than $\P(S^n_\varepsilon<T)/2$ for every $\xi<\xi_0$. (In case $\P(S^n_\varepsilon<T)=0$, the proof is done and this is also true). Thus, for $0<\xi<\xi_0$, the left hand side in \eqref{etape2} is lower bounded by $\eta(\xi)\P(S^n_\varepsilon<T)$. This proves \eqref{reste_a_prouver} provided $\eta$ is positive, which we shall establish in Step 5.\\

\noindent \textbf{Step 5} Let us prove that $\eta$ defined in \eqref{reste_a_prouver2} is positive. Since it is a minimum over a finite number of terms, let us consider one of the latter. For this, we consider $(z,r)\in \D_{\R^d}$ and our purpose is to prove that
$$\P_{z,r}\Big(\inf_{u\in [0,T]}\widetilde{\mathcal{Z}}_u\geq \frac{3\varepsilon}{4}\Big)>0.$$
For $M>0$, let us define the stopping time $\varsigma_M=\inf\{t\geq 0,\, \widetilde{\mathcal{Z}}\geq M\}$ and let us introduce
\begin{equation}
\tau_{\varepsilon/2}=\inf\big\{t\geq 0,\, \widetilde{\mathcal{Z}}_t\leq \frac{\varepsilon}{2}\big\}
\end{equation}such that $\P_{z,r}(\inf_{s\in [0,T]}\widetilde{\mathcal{Z}}_s>\varepsilon/2)= \P_{z,r}(\tau_{\varepsilon/2}> T)$. Our purpose is to prove that the latter quantity is positive. We will change the law of the process to be able to compare this quantity to the one defined with a Brownian motion. Let $\lambda>0$. From Itô's formula:
\begin{align*}
e^{\lambda\widetilde{\mathcal{Z}}_{t\wedge \varsigma_M}}= & e^{\lambda\varepsilon}+\int_0^{t\wedge \varsigma_M} \big(\lambda^2\bar{R}-\lambda(\bar{D}+\bar{U}N)\big)\widetilde{\mathcal{Z}}_s e^{\widetilde{\mathcal{Z}}_s}ds+\int_0^{t\wedge \varsigma_M}\lambda \sqrt{\rho(s)\widetilde{\mathcal{Z}}_s}e^{\widetilde{\mathcal{Z}}_s}dB_s.
\end{align*}Taking the expectation and choosing $N$ sufficiently large ($N>(\lambda\bar{R}-\bar{D})/\bar{U}$), we obtain that
$\E\big(\exp(\lambda\widetilde{\mathcal{Z}}_{t\wedge \varsigma_M})\big)\leq \exp(\lambda \varepsilon)$. From \eqref{def_mathcalN} and since $2\underline{R}\leq \rho(u)\leq 2\bar{R}$, we can classically show that $\E\big(\sup_{t\in [0,T]}\widetilde{\mathcal{Z}}^2_t\big)<+\infty$, from which we deduce that $\lim_{M\rightarrow +\infty}\varsigma_M=+\infty$ a.s. Moreover, it follows by Fatou's lemma that for any $t\in [0,T]$, $\E\big(\exp(\lambda\widetilde{\mathcal{Z}}_{t})\big)\leq \exp(\lambda \varepsilon)$ and by Jensen's inequality and Fubini's theorem, we get
\begin{equation}
\E\Big(e^{ \int_0^T \frac{(\bar{D}+\bar{U}N)^2}{2\rho(s)}\widetilde{\mathcal{Z}}_s ds}\Big)\leq \frac{1}{T}\int_0^T \E\big(e^{ \frac{(\bar{D}+\bar{U}N)^2T}{2\underline{R}}\widetilde{\mathcal{Z}}_s}\big)ds\leq e^{ \frac{(\bar{D}+\bar{U}N)^2T \varepsilon}{2\underline{R}}}<+\infty.
\end{equation}Novikov's criterion is satisfied and  Girsanov's Theorem applied to \eqref{def_mathcalN} tells us that under the probability $\mathbb{M}$ defined by
\begin{equation}\frac{d\mathbb{M}}{d\P_{z,r}}\Big|_{\mathcal{F}_t}=\exp\Big(-\int_0^t \frac{(\bar{D}+\bar{U}N)\sqrt{\widetilde{\mathcal{Z}}_{u}}}{\sqrt{\rho(u)}}dB_u-\frac{1}{2}\int_0^t \frac{(\bar{D}+\bar{U}N)^2 \widetilde{\mathcal{Z}}_u}{\rho(u)}du\Big),
\end{equation}$\widetilde{\mathcal{Z}}$ is a Brownian motion started at $\varepsilon$.
Then we have
\begin{align*}
 & \mathbb{M}(\tau_{\varepsilon/2}\leq T)+\mathbb{M}(\tau_{\varepsilon/2}> T)=1\\
 & \E^\mathbb{M}\big(\widetilde{\mathcal{Z}}_{T\wedge \tau_{\varepsilon/2}}\big)= \frac{\varepsilon}{2}\mathbb{M}(\tau_{\varepsilon/2}\leq T)+\E^\mathbb{M}\big(\widetilde{\mathcal{Z}}_{T}\ind_{\tau_{\varepsilon/2}> T}\big)=\varepsilon.
\end{align*}If $\mathbb{M}(\tau_{\varepsilon/2}\leq T)=1$, this yields thus a contradiction since we would obtain $\varepsilon/2=\varepsilon$ for the second equation. Thus, $\mathbb{M}(\tau_{\varepsilon/2}\leq T)<1$ and $\mathbb{P}_{z,r}(\tau_{\varepsilon/2}\leq T)<1$. This shows that $\eta>0$.\\

\noindent \textbf{Step 6} It now remains to prove \eqref{reste_a_prouver3}. We follow Dawson and Perkins \cite[Lemma 7.3]{dawsonperkinsAMS}. For $n\in \N^*$, we can exhibit, by a coupling argument, a process $\widetilde{X}^n$ constituted of independent particles with birth rate $n r(t,y)+b(t,y)$ and death rate $nr(t,y)$, started at $X^n_0$ and which dominates $X^n$. In particular, for $T>0$ and for any compact set $K\subset \D_{\R^d}$, $\E(X^n_T((K^T)^c))\leq \E(\widetilde{X}^n_T((K^T)^c))$. \\
The tree underlying $\widetilde{X}^n$ can be obtained by pruning a Yule tree with traits in $\R^d$, where a particle of lineage $y$ at time $t$ gives two offspring at rate $2n r(t,y)+b(t,y)$. One has lineage $y$ and the other has lineage $(y|t|y+h)$ where $h$ is  distributed following $K^n(dh)$. Using Harris-Ulam-Neveu's notation to label the particles (see e.g. Dawson \cite{dawson}), we denote by $Y^{n,\alpha}$ for $\alpha\in \mathcal{I}=\bigcup_{m=0}^{+\infty}\{0,1\}^{m+1}$ the lineage of the particle with label $\alpha$. Particles are exchangeable and the common distribution of the process $Y^{n,\alpha}$ is the one of a pure jump process of $\R^d$, where the jumps occur at rate $2n r(t,y)+b(t,y)$ and where the jump sizes are distributed according to the probability measure $\frac{1}{2}\delta_0(ds)+\frac{1}{2}K^n(dh)$. We denote by $\,\P^n_x\,$ its distribution starting from $x\in \R^d$. It is standard to prove that the family of laws $(\P^n_{x},\, n\in \N^*,\, x\in C)$ is tight as soon as $C$ is a compact set of $\R^d$.\\
At each node of the Yule tree, an independent pruning is made: the offspring are kept with probability $(n r(t,y)+b(t,y))/(2 n r(t,y)+b(t,y))$ and are erased with probability $n r(t,y)/(2 n r(t,y)+b(t,y))$. Let us denote by $V_t$ the set of individuals alive at time $t$ and write $\alpha\succ i$ to say that the individual $\alpha$ is a descendant of the individual $i$:
 \begin{align*}
 \E\big(\widetilde{X}^n_T((K^T)^c)\big)= & \E\left(\frac{1}{n}\sum_{i=1}^{N^n_0} \sum_{\alpha\succ i} \E\Big(\ind_{Y^{n,\alpha}\notin K^T}\ind_{\alpha \in V_T}\Big)\right)
 =  \E\left(\frac{1}{n} \sum_{i=1}^{N^n_0} \P^n_{X^i_0}\big((K^T)^c\big)\E\Big(\sum_{\alpha \succ i}\ind_{\alpha \in V_T}\Big)\right)\\
 \leq & \E\left(\frac{1}{n} \sum_{i=1}^{N^n_0} \P^n_{X^i_0}\big((K^T)^c\big)e^{\bar{B}T}\right)
 =  e^{\bar{B}T}\int_{\D_{\R^d}}X^n_0(dy)\P^n_{y_0}\big((K^T)^c\big).
 \end{align*}
 The bound  $e^{\bar{B}T}$ is an upper bound of the mean population size at $T$ that descends from a single initial individual, when the growth rate is bounded above by $\bar{B}$.
 For each $\varepsilon>0$ there exists a compact set $C$ of $\R^d$ and a compact set $K$ of $\D_{\R^d}$ such that
 $$\sup_{n\in \N^*}X^n_0(C^c)\leq \varepsilon\quad \mbox{ and }\quad \sup_{n\in \N^*}\sup_{y_0\in C}\P^n_{y_0}\big((K^T)^c\big)\leq \varepsilon,$$
 which concludes the proof.
\end{proof}

\subsection{Identification of the limiting values}\label{sectionidentification}

Let us denote by $\bar{X}\in \Co(\R_+,\mathcal{M}_F(\D))$ a limit point of $(\bar{X}^n)_{n\in \N^*}$.
Our purpose here is to characterize the limiting value via the martingale problem that appears in Theorem \ref{propconvergence}.\\
Notice that the limiting process $\bar{X}$ is necessarily almost surely continuous as
\begin{equation}
\sup_{t\in \R_+}\sup_{\varphi,\,\|\varphi\|_\infty\leq 1} |\langle \bar{X}^n_t,\varphi\rangle -\langle \bar{X}^n_{t_-},\varphi\rangle |\leq \frac{1}{n}.\label{la_limite_est_continue}
\end{equation}

For the proof of Theorem \ref{propconvergence}, we will need the following Proposition, which establishes the uniqueness of the solution of
(\ref{martingalesecondcas})-(\ref{crochetmartingalesecondcas}). Since the limiting process takes its values in $\Co([0,T],\mathcal{M}_F(\Co(\R_+,\R^d)))$, we will use the test functions in \eqref{testfunction:dawson} instead of \eqref{testfunction}.

\begin{prop}\label{prop_unique}
(i) The solutions of the martingale problem \eqref{martingalesecondcas}-\eqref{crochetmartingalesecondcas} also solve the following martingale problem, where $\varphi$ is a test function of the form \eqref{testfunction:dawson} and $\widetilde{\Delta}$ has been defined in \eqref{laplacien}:
\begin{multline}
M^{\varphi}_t=  \langle \bar{X}_t,\varphi\rangle -\langle \bar{X}_0,\varphi\rangle-\int_0^t \int_{\D_{\R^d}} \Big(
p\, r(s,y)\frac{\sigma^2}{2}\ \widetilde{\Delta}\varphi(s,y)
+\gamma(s,y,\bar{X}^s)\varphi(y)\Big)\bar{X}_s(dy)\, ds\label{martingalecasdawson}
\end{multline}is a square integrable martingale with quadratic variation
\begin{align}
\langle M^\varphi\rangle_t= \int_0^t \int_{\D_{\R^d}}2\,r(s,y) \varphi^2(y)\bar{X}_s(dy)\, ds.\label{crochetmartingalecasdawson}
\end{align}
(ii) There exists a unique solution to the martingale problem \eqref{martingalecasdawson}-\eqref{crochetmartingalecasdawson}.\\
(iii) There exists a unique solution to the martingale problem \eqref{martingalesecondcas}-\eqref{crochetmartingalesecondcas}.
\end{prop}

\noindent In the course of the proof, we will need the following lemma, whose proof uses standard arguments with $r(t,y)$ depending on all the trajectory (see \eqref{def:R}).
\begin{lemme}\label{lemme-SDE}
Let us consider the following stochastic differential equation on $\R^d$ driven by a standard Brownian motion $B$:
\begin{equation}
 Y_t= Y_0+\int_0^t \sqrt{\sigma^2 p \, r(s,Y^s)}dB_s.\label{def:Yt}
\end{equation}There exists a unique solution to \eqref{def:Yt}.
\end{lemme}

\medskip
\begin{proof}[Proof of Proposition \ref{prop_unique}]It is clear that (iii) follows from (i) and (ii). \\
Let us begin with the proof of (i).
We consider a function $\varphi$ of the form \eqref{testfunction:dawson} and we assume without restriction that the functions $g_j$'s are positive. Firstly, let us show that (\ref{martingalecasdawson}) defines a martingale. Proceeding as in the proof of Theorem 5.6 in \cite{fourniermeleard}, for $k\in \N^*$, $0\leq s_1\leq \dots s_k<s<t$ and  $\phi_1,\dots \phi_k$
bounded measurable functions on $\mathcal{M}_F(\D_{\R^d})$, we define for $X\in \D(\R_+,\mathcal{M}_F(\D_{\R^d}))$ the function
\begin{multline}
\Psi(X)=\phi_1(X_{s_1})\dots \phi_k(X_{s_k})\Big\{\langle X_t,\varphi\rangle-\langle X_s,\varphi\rangle-\int_s^t \int_{\D_{\R^d}}
\big(p r(u,y)\frac{\sigma^2}{2} \widetilde{\Delta}\varphi(y)\\
+\gamma(u,y,\bar{X}^u)\big)\varphi(y)\Big)\bar{X}_u(dy)\, du
\Big\}.
\end{multline}
Let us prove that $\E(\Psi(\bar{X}))=0$. We consider, for $q\in \N^*$, the test functions $\varphi_{q}(y)=G_{g_q}(y)$ with $G(x)=\exp(x)$ and $g_q(s,x)=\sum_{j=1}^m \log(g_j(x))k^q(t_j-s)$ where $k^q(x)$ is the density of the Gaussian distribution with mean 0 and variance $1/q$. From (\ref{martingalesecondcas})-\eqref{crochetmartingalesecondcas} and \eqref{estimee_momentp}, the process
\begin{multline}
M^{\varphi_{q}}_t=  \langle \bar{X}_t,\varphi_{q}\rangle -\langle \bar{X}_0,\varphi_{q}\rangle-\int_0^t \int_{\D_{\R^d}} \Big(
p\, r(s,y)\frac{\sigma^2}{2}\ \mathcal{D}^2\varphi_{q}(s,y)\\
+\gamma(s,y,\bar{X}^s)\varphi_{q}(y)\Big)\bar{X}_s(dy)\, ds\label{etape33}
\end{multline}is a square integrable martingale, hence uniformly integrable. The latter property together with Lemma \ref{lien:fonctionstest} implies that $\phi_1(X_{s_1})\dots \phi_k(X_{s_k})(M^{\varphi_q}_t-M^{\varphi_q}_s)$ converges to $\Psi(\bar{X})$ in $L^1$ and that $\E(\Psi(\bar{X}))=0$.\\
Let us show that the bracket of $M^\varphi$ is given by \eqref{crochetmartingalecasdawson}. By a similar argument as previously, we firstly check that the  process
\begin{multline}
\langle \bar{X}_t,\varphi\rangle^2-\langle \bar{X}_0,\varphi\rangle^2 -\int_0^t \int_{\D} 2r(s,y) \varphi^2(y) \bar{X}_s(dy)ds-  \int_0^t 2 \langle \bar{X}_s,\varphi\rangle
\int_{\D}\Big(p r(s,y)\frac{\sigma^2}{2}\widetilde{\Delta}\varphi(y)\\
  +\gamma(s,y,\bar{X}^s)\varphi(y)\Big)\bar{X}_s(dy)\, ds
\label{etape25}
\end{multline}
is a martingale. In another way, using Itô's formula and (\ref{martingalecasdawson}),
\begin{multline}
\langle \bar{X}_t,\varphi\rangle^2-\langle \bar{X}_0,\varphi\rangle^2 - \langle M^\varphi\rangle_t - \int_0^t 2 \langle \bar{X}_s,\varphi\rangle
\int_{\D}\Big(p r(s,y)\frac{\sigma^2}{2}\widetilde{\Delta}\varphi(y)\\
 +\gamma(s,y,\bar{X}^s)\varphi(y)\Big)\bar{X}_s(dy)\, ds
\label{etape26}
\end{multline}is a martingale. Comparing (\ref{etape25}) and (\ref{etape26}), we obtain (\ref{crochetmartingalecasdawson}).\\

\noindent Let us now prove (ii). Let $\P$ be a solution of the martingale problem \eqref{martingalesecondcas}-\eqref{crochetmartingalesecondcas} and let $\bar{X}$ be the canonical process on $\Co(\R_+,\mathcal{M}_F(\D_{\R^d}))$. We first use Dawson-Girsanov's theorem (see \cite[Section 5]{dawsonGeostocCalculus}, \cite[Theorem 2.3]{evansperkins}) to get rid of the non-linearities. This theorem can be applied since
\begin{align*}
\E\Big(\int_0^T \int_{\D_{\R^d}} \gamma^2(s,y,\bar{X}^s)\bar{X}_s(dy)\ ds\Big)<+\infty.
\end{align*}Indeed, $\gamma$ is bounded and $\E\Big(\sup_{t\in [0,T]}\langle \bar{X}_t,1\rangle^2\Big)<+\infty$ by taking \eqref{estimee_momentp} to the limit.
Hence, there exists a probability measure $\Q$ on $\Co(\R_+,\mathcal{M}_F(\D_{\R^d}))$ equivalent to $\P$ such that under $\Q$, and for every test function $\varphi$ of the form \eqref{testfunction:dawson}, the process
\begin{align}
\widetilde{M}^{\varphi}_t=\langle \bar{X}_t,\varphi\rangle-\langle \bar{X}_0,\varphi\rangle-\int_0^t \int_{\D_{\R^d}}\frac{pr(s,y)\sigma^2}{2}\widetilde{\Delta}\varphi(s,y)\bar{X}_s(dy)ds\label{defMtilde}
\end{align}is a martingale with quadratic variation \eqref{crochetmartingalesecondcas}. Thus, if there is uniqueness of the probability measure $\mathbb{Q}$ which solves the martingale problem \eqref{defMtilde}-\eqref{crochetmartingalesecondcas} we will deduce the uniqueness of the solution of the martingale problem \eqref{martingalesecondcas}-\eqref{crochetmartingalesecondcas}. \\

\noindent Let us now prove that the Laplace transform of $\bar{X}$ under $\Q$ is uniquely characterized using the solution $Y$ of the stochastic differential equation  \eqref{def:Yt}. We associate with $Y$ its path-process $W\in \Co(\R_+,\Co(\R_+,\R^d))$ defined by
\begin{equation}
W_t=(Y_{t\wedge s})_{s\in \R_+}.\label{def:W}
\end{equation}
The path-process $W$ is not homogeneous 
 but it is however a strong Markov process with semigroup defined for all $s\leq t$ and all $\varphi\in \Co_b(\Co(\R_+,\R^d),\R)$ by
\begin{equation}
S_{s,t}\varphi(y)=\E^\mathbb{Q}\big(\varphi(W_t)\, |\, W_s=y^s\big).
\end{equation}Moreover, the infinitesimal generator $\widetilde{A}$ of $W$ at time $t$ is defined for all $\varphi$ as in \eqref{testfunction:dawson} by
\begin{equation}
\widetilde{A}\varphi(t,y)=\frac{p\sigma^2}{2}r(t,y)\widetilde{\Delta}\varphi(t,y).
\end{equation}
Then it can be shown that the log-Laplace functional of $\bar{X}_t$ under the probability $\mathbb{Q}$, $L(s,t,y,\varphi)=\E^\mathbb{Q}\big(\exp(-\langle \bar{X}_t,\varphi\rangle)\, |\, \bar{X}_s=\delta_{y^s}\big)$ satisfies
\begin{equation}
L(s,t,y,\varphi)=e^{- V_{s,t}\varphi(y)},
\end{equation}where $V_{s,t}\varphi(y)$ solves:
\begin{align}
V_{s,t}\varphi(y)= & \E\Big(\varphi(W_t)-\int_s^t \frac{p\sigma^2}{2}r(u,W_u)\big(V_{u,t}\varphi(W_u)\big)^2 du \, |\, W_s=y^s\Big)\nonumber\\
= & S_{s,t}\varphi(y)-\int_s^t \frac{p\sigma^2}{2}S_{s,u}\Big(r(u,.)\big(V_{u,t}\varphi(.)\big)^2\Big)(y) du.\label{eq:cumulant}
\end{align}
Adapting Theorem 12.3.1.1 of \cite[p.207]{dawson}, there exists a unique solution to \eqref{eq:cumulant}.
Indeed, let $V^1$ and $V^2$ be two solutions. From \eqref{eq:cumulant}, we see that for $i\in \{1,2\}$,
\begin{equation}
\sup_{s,t,y}|V^i_{s,t}\varphi(y)|\leq \sup_{y}|\varphi(y)|=\| \varphi\|_\infty.
\end{equation}We have
\begin{align*}
|V^2_{s,t}\varphi(y)-V^1_{s,t}\varphi(y)|= & \Big|\frac{p\sigma^2}{2} \int_s^t S_{s,u}\Big(r(u,.)\big((V^2_{u,t}\varphi(.))^2-(V^1_{u,t}\varphi(.))^2\big)\Big)(y) du\Big|\nonumber\\
\leq & \frac{p\sigma^2}{2} \int_s^t S_{s,u}\Big(r(u,.)2\| \varphi\|_\infty \big|V^2_{u,t}\varphi(.)-V^1_{u,t}\varphi(.)\big|\Big)(y) du\nonumber\\
\leq & p\sigma^2\| \varphi\|_\infty\bar{R}\int_s^t S_{s,u}\Big(\big|V^2_{u,t}\varphi-V^1_{u,t}\varphi\big|\Big)(y) du.
\end{align*} Uniqueness follows from the  Dynkin's generalized Gronwall inequality (see \eg \cite[Lemma 4.3.1]{dawson}).\\
 In conclusion, the Laplace transform of $\bar{X}_t$ is uniquely characterized for every $t>0$ by\\ $
\E_{\bar{X}_0}\big(\exp(-\langle \bar{X}_t,\varphi\rangle)\big)=\exp(-\langle \bar{X}_0,V_{0,t}\varphi\rangle)$.
Thus, the one-marginal distributions of the martingale problem \eqref{defMtilde}-\eqref{crochetmartingalesecondcas} are uniquely determined and thus, there exists a unique solution to \eqref{defMtilde}-\eqref{crochetmartingalesecondcas}.
\end{proof}

\bigskip
\noindent It is now time to turn to the
\begin{proof}[Proof of Theorem \ref{propconvergence}]
Let $\bar{X}\in \Co(\R_+,\mathcal{M}_F(\D_{\R^d}))$ be a limiting process of the sequence $(\bar{X}^n)_{n\in \N^*}$ and let us denote  again by $(\bar{X}^n)_{n\in \N^*}$
the subsequence that converges in law to $\bar{X}$. Since the limiting process is continuous, the convergence holds in $\D(\R_+,\mathcal{M}_F(\D_{\R^d}))$
embedded with the Skorohod topology, but also with the uniform topology for all $T>0$ (\eg \cite{billingsley}).\\
The aim is to identify the martingale problem solved by the limiting value $\bar{X}$. We will see that it satisfies \eqref{martingalesecondcas}-\eqref{crochetmartingalesecondcas} which admits a unique solution by Proposition \ref{prop_unique}. This will conclude the proof.\\
Firstly, we show that (\ref{martingalesecondcas}) defines a martingale. For $k\in \N^*$, let  $0\leq s_1\leq \dots s_k<s<t$ and let $\phi_1,\dots \phi_k$ be
bounded measurable functions on $\mathcal{M}_F(\D_{\R^d})$. Let $G \in \Co^2_b(\R,\R)$, $g\in \Co^{0,2}_b(\R_+\times \R^d,\R)$ and $G_g$ be functions as in \eqref{testfunction}. We define for $X\in \D(\R_+,\mathcal{M}_F(\D_{\R^d}))$ the function
\begin{multline}
\Phi(X)=\phi_1(X_{s_1})\dots \phi_k(X_{s_k})\Big\{\langle X_t,G_g\rangle-\langle X_s,G_g\rangle-\int_s^t \int_{\D_{\R^d}}
\Big(p r(u,y)\frac{\sigma^2}{2} \mathcal{D}^2G_g(u,y)\\
+\gamma(u,y, X^u)G_g(y)\Big)X_u(dy)\, du
\Big\}.
\end{multline}Let us prove that $\E(\Phi(\bar{X}))=0$. From (\ref{martingalesecondcas}),
\begin{align}
0= & \E\Big(\phi_1(\bar{X}^n_{s_1})\dots \phi_k(\bar{X}^n_{s_k})\big(M^{n,G_g}_t-M^{n,G_g}_s\big)\Big)\nonumber\\
= & \E\big(\Phi(\bar{X}^n)\big)
+\E\Big(\phi_1(\bar{X}^n_{s_1})\dots \phi_k(\bar{X}^n_{s_k})\big(A_n+B_n\big)\Big),\label{etape1}
\end{align}where
\begin{align}
A_n= & \int_s^t \int_{\D_{\R^d}} r(u,y) \Big\{n \Big(\int_{\R^d}G_g(y|u|y_u+h)K^n(dh)-G_g(y)\Big)\label{def:AnBn}\\
 &\qquad  \qquad \qquad - \frac{p\sigma^2}{2} \mathcal{D}^2G_g(u,y)  \Big\}\bar{X}^n_u(dy)\, du\nonumber\\
B_n=  & \int_s^t\int_{\D_{\R^d}}
b(u,y)\int_{\R^d}\Big(G_g(y|u|y_u+h)-G_g(y) \Big) K^n(dh)\,\bar{X}^n_u(dy)\, du.\nonumber
\end{align}
As $\bar{X}$ is continuous, $\Phi$ is a.s. continuous at $\bar{X}$. Moreover, from
$|\Phi(X)|\leq C\big(\sup_{s\leq t}\langle X_s,1\rangle+\sup_{s\leq t}\langle X_s,1\rangle^2\big)$ and Prop. \ref{prop_usefulresults}, we deduce that $(\Phi(\bar{X}^n))_{n\in \N^*}$ is a uniformly integrable sequence such that
\begin{align}
\lim_{n\rightarrow +\infty}\E\big(\Phi(\bar{X^n})\big)=\E\big(\Phi(\bar{X})\big).\label{etape4}
\end{align}
Estimates \eqref{estimee_momentp} and \eqref{limitD2} imply that $A_n$ and $B_n$ defined in \eqref{def:AnBn} converge in $L^1$ to zero. Hence \eqref{etape1} and \eqref{etape4} entail the desired result. The computation of the bracket of the martingale $M^{G_g}$ is standard (see for instance \cite{fourniermeleard}).\\
 The proof  of Theorem \ref{propconvergence} is now complete.
\end{proof}

\section{Distributions of the genealogies}

An important question is to gather information on the lineages of individuals alive in the population.

\medskip \noindent
Firstly, remark in Proposition \ref{prop_unique} we obtain a martingale problem introduced by Perkins \cite[Th. 5.1 p.64]{perkinsAMS} and leading to the following representation result: under $\Q$, $(\bar{X}_t)_{t\in \R_+}$ has the same distribution as the historical superprocess $(K_t)_{t\in \R_+}$ which is the unique solution of
\begin{align}
& Y_t(y)= Y_0(y)+ \int_0^t \sqrt{\sigma^2 p r(s,Y^s(y))}dy_s\\
& K_0=\bar{X}_0,\qquad \langle K_t,\varphi\rangle= \int_{\D_{\R^d}}\varphi(Y(y)^t)H_t(dy)\label{perkinsrepresentation}
\end{align}where $(H_t(dy))_{t\in \R_+}$ is under $\mathbb{Q}$ a historical Brownian superprocess (see \cite{dawsonperkinsAMS}), and for $\varphi$ in a sufficiently large class of test functions, of the form \eqref{testfunction:dawson} for instance. The stochastic integral $\int_0^t \sqrt{\sigma^2 p r(s,Y^s(y))}dy_s$ is the H-historical integral constructed in \cite[Th. 2.12]{perkinsAMS}.

\subsection{Lineages drawn at random}
For $t>0$, $\bar{X}_t$ is a random measure on $\D_{\R^d}$ and its restriction to $\D([0,t],\R^d)$ correctly renormalized gives the distribution of the lineage of an individual chosen at random at time $t$. Let us define $\mu_t(dy)=\bar{X}_t(dy)/\langle \bar{X}_t,1\rangle$ such that for any measurable test function $\varphi$ on $\D_{\R^d}$:
\begin{equation}
\langle \mu_t,\varphi\rangle=\frac{\langle \bar{X}_t,\varphi\rangle}{\langle \bar{X}_t,1\rangle}.\label{def:mu_t}
\end{equation}
For instance, choosing $\varphi(y)=\ind_{A}(y)$ for a measurable subset $A\subset \D_{\R^d}$, we obtain the proportion of paths belonging to $A$ under the random probability measure $\mu_t$.
Studying such random probability measure remains unfortunately a difficult task and we will also consider its intensity probability measure $\E\mu_t$ defined for any test function $\varphi$ as $\langle \E\mu_t,\varphi\rangle=\E(\langle \mu_t,\varphi\rangle)$. This approach has been used in cases where the branching property holds (for instance in \cite{bansayedelmasmarsalletran}). The hope is to identify the infinitesimal generator of Markov processes that will give us the average distribution of a path chosen at random.

\begin{prop}For $t>0$, a test function $\varphi$ as in \eqref{testfunction:dawson} and $\mu_t$ defined in \eqref{def:mu_t}:
\begin{align}
\E\big(\langle \mu_t,\varphi\rangle\big)= & \E\big(\langle \mu_0,\varphi\rangle\big)+\E\Big(\int_0^t \big\langle \mu_s, \frac{pr(s,.)\sigma^2}{2} \widetilde{\Delta}\varphi(s,.)\big\rangle ds\Big)\nonumber\\
 +& \E\Big(\int_0^t \big(\langle \mu_s,\gamma(s,.,\bar{X}^s)\varphi(.)\rangle-\langle \mu_s,\varphi\rangle \langle \mu_s, \gamma(s,.,\bar{X}^s)\rangle \big) ds\Big)\nonumber\\
 + & \E\Big(\int_0^t \frac{2}{\langle \bar{X}_s,1\rangle}\big(\langle \mu_s,\varphi\rangle \langle \mu_s,r(s,.)\rangle - \langle \mu_s,r(s,.)\varphi(.)\rangle \big) ds\Big)\label{eq:Emu_t}
\end{align}
\end{prop}

Remark that the moment equation \eqref{eq:Emu_t} cannot be closed due to the nonlinearity.

\begin{proof}
We consider \eqref{martingalesecondcas}-\eqref{crochetmartingalesecondcas} and Itô's formula:
\begin{align*}
\E\big(\langle  & \mu_t,\varphi\rangle\big)=
\langle \mu_0,\varphi\rangle+\E\Big(\int_0^t \int_{\D_{\R^d}} \frac{\frac{1}{2}pr(s,y)\sigma^2 \widetilde{\Delta}\varphi(s,y)+\gamma(s,y,\bar{X}^s)\varphi(y)}{\langle \bar{X}_s,1\rangle} \bar{X}_s(dy)\,ds\Big)\\
- & \E\Big(\int_0^t \int_{\D_{\R^d}} \frac{\langle \bar{X}_s,\varphi\rangle}{\langle \bar{X}_s,1\rangle^2}\gamma(s,y,\bar{X}^s)\bar{X}_s(dy)\,ds\Big)\\
+ & \E\Big(\frac{1}{2} \Big[\int_0^t \int_{\D_{\R^d}} \frac{2\langle \bar{X}_s,\varphi\rangle}{\langle \bar{X}_s,1\rangle^3}2r(s,y)  \bar{X}_s(dy)\,ds
-2\int_0^t \int_{\D_{\R^d}} \frac{1}{\langle \bar{X}_s,1\rangle^2}2r(s,y) \varphi(y)\bar{X}_s(dy)\, ds\Big]\Big)\\
= & \E\big(\langle \mu_0,\varphi\rangle\big)+\E\Big(\int_0^t \big\langle  \mu_s, \frac{1}{2}pr(s,.)\sigma^2 \widetilde{\Delta}\varphi(s,.)+\gamma(s,.,\bar{X}_s)\varphi(.)\big\rangle ds\Big)\\
+ &   \E\Big(\int_0^t \langle \mu_s,\varphi\rangle \langle \mu_s, \gamma(s,.,\bar{X}^s)\rangle ds\Big)
+  \E\Big(\int_0^t \frac{\langle \mu_s,\varphi\rangle \langle \mu_s, 2r(s,.)\rangle}{\langle \bar{X}_s,1\rangle}ds
-\int_0^t  \frac{\langle \mu_s,2r(s,.) \varphi\rangle}{\langle \bar{X}_s,1\rangle} ds\Big)
\end{align*}This ends the proof.
\end{proof}

\medskip \noindent
In \eqref{eq:Emu_t}, we recognize two covariance terms under the probability measure $\mu_s$:
\begin{align}
& \Cov_{\mu_s}(\gamma(s,.,\bar{X}_s),\varphi)=\langle \mu_s,\gamma(s,.,\bar{X}_s)\varphi(.)\rangle-\langle \mu_s,\varphi\rangle \langle \mu_s, \gamma(s,.,\bar{X}_s)\rangle \nonumber\\
& \Cov_{\mu_s}(r(s,.),\varphi)= \langle \mu_s,r(s,.)\varphi(.)\rangle-\langle \mu_s,\varphi\rangle \langle \mu_s,r(s,.)\rangle.\label{termes-cov}
\end{align}
The covariance terms \eqref{termes-cov} can be viewed as the generators of jump terms. For example,
\begin{align*}
\Cov_{\mu_s}(\gamma(s,.,\bar{X}_s),\varphi)=\langle \mu_s,\gamma(s,.,\bar{X}_s)\rangle \int_{\D_{\R^d}}\Big(\int_{\D_{\R^d}} \varphi(z) \frac{\gamma(s,z,\bar{X}_s)\mu_s(dz)}{\langle \mu_s,\gamma(s,.,\bar{X}_s)\rangle} - \varphi(y)\Big)\mu_s(dy).
\end{align*}Heuristically, the distribution of the path of a particle chosen at random at time $t$ is as follows. We consider diffusive particles with generator $pr(s,.)\sigma^2 \widetilde{\Delta}$. At rate $2\langle \mu_s,r(s,.)\rangle/\langle \bar{X}_s,1\rangle$,  the particles are resampled in the distribution $\mu_s$ biased by the function $r(s,.)$ where individuals with larger allometric coefficients $r(s,.)$ are more likely to be chosen. At rate $\langle \mu_s,\gamma(s,y,\bar{X}^s)\rangle$, particles are resampled in the distribution $\mu_s$ biased by the function $\gamma(s,.,\bar{X}^s)$ which gives more weight to particles having higher growth rate.

\subsection{Case of constant allometric function and growth rate}\label{section:r-gamma-constants}

A particular case is when the allometric function $r(s,y)$ and the growth rate $\gamma(s,y,\bar{X}^s)$ are constant and equal to $\bar{R}$ and $\bar{\gamma}$ respectively. In particular, a null growth rate $\bar{\gamma}=0$ models the ``neutral''case.
 In that case, we will show  that the process of interest corresponds to a Fleming-Viot process with genealogies similar to the one introduced in \cite{grevenpfaffelhuberwinter}. Notice that this bears similarities with the connections established between superprocesses and (classical) Fleming-Viot processes (e.g. Perkins \cite{perkins-SSP}).\\
Let us introduce the notation
\begin{equation}
\langle \mu_t\otimes \mu_t, \chi(t,.,.)\rangle=\int_{\D_{\R^d}}\int_{\D_{\R^d}} \chi(t,y,z)\mu_t(dy)\mu_t(dz).
\end{equation}
This quantity can help to describe the most recent commun ancestor (MRCA) of two individuals chosen at random in the population.

\begin{prop}\label{propFV}If $r(s,y)=\bar{R}$ and $\gamma(s,y,X)=\bar{\gamma}$, then\\
(i)$\Cov_{\mu_s}(r(s,.),\varphi)=0$ and $\Cov_{\mu_s}(\gamma(s,.,\bar{X}_s),\varphi)=0$.\\
(ii) The lineage distributions under $\E(\mu_t)$ are Brownian motions.\\
(iii) Let $\chi(y,z)$ be a function of two variables such that $y\mapsto \chi(y,z)$ and $z\mapsto \chi(y,z)$ are of the form \eqref{testfunction:dawson}. For test functions $\phi(\mu)=\langle \mu\otimes \mu,\chi\rangle/\langle \mu,1\rangle^2$, the generator of $\bar{X}$ becomes
 \begin{align*}
 L^{FV}\phi(X)=\frac{p\bar{R}\sigma^2}{2}\Big\langle \frac{X}{\langle X,1\rangle}\otimes \frac{X}{\langle X,1\rangle}, \widetilde{\Delta}^{(2)}\chi\Big\rangle+\frac{2\bar{R}}{\langle X,1\rangle}\Big(\int_{\D_{\R^d}}\chi(y,y)\frac{X(dy)}{\langle X,1\rangle}-\langle \frac{X}{\langle X,1\rangle}\otimes \frac{X}{\langle X,1\rangle},\chi\rangle\Big),
 \end{align*}where $\widetilde{\Delta}^{(2)}\chi(y,z)=\widetilde{\Delta}(y\mapsto \chi(y,z))+\widetilde{\Delta}(z\mapsto \chi(y,z))$.
\end{prop}

\medskip
\noindent In the expression of the generator $L^{FV}$, we recognize a resampling operator, where the resampling occurs nonlinearly at the  rate $2\bar{R}/\langle X,1\rangle$.

\begin{proof}[Proof of Proposition \ref{propFV}]
The proof of (i) is obvious. To prove (ii), let us remark that for any test function $\varphi$ of the form \eqref{testfunction:dawson}, \eqref{eq:Emu_t} yields
\begin{align*}
\langle \E(\mu_t),\varphi\rangle= \langle \E(\mu_0),\varphi\rangle+\int_0^t \langle \E(\mu_s),p\bar{R}\sigma^2 \widetilde{\Delta}\varphi(s,.)\rangle ds.
\end{align*}Therefore, under $\E(\mu_t)$, the lineages have the same finite-dimensional distributions as Brownian motions with diffusion coefficient $p\bar{R}\sigma^2$.\\
Let us now prove (iii). Let $\varphi$ and $\psi$ be two  functions of the form \eqref{testfunction:dawson}. Using Itô's formula,
\begin{multline}
\langle \mu_t,\varphi\rangle \langle \mu_t,\psi\rangle= \langle \mu_0,\varphi\rangle \langle \mu_0,\psi\rangle+M_t^{\varphi,\psi}\label{etape29}\\
\begin{aligned}
+ & \int_0^t \Big(\frac{p\bar{R}\sigma^2}{2} \langle \mu_s\otimes \mu_s,\widetilde{\Delta}(\varphi \psi)(s,.)\rangle
  +  2\bar{R}  \int_{\R} \frac{(\varphi(y)-\langle \mu_s,\varphi\rangle)(\psi(y)-\langle \mu_s,\psi\rangle)}{\langle \bar{X}_s,1\rangle}\mu_s(dy)\Big)ds,
\end{aligned}
\end{multline}where $M^{\varphi,\psi}$ is a square integrable martingale.\\
For a function $\chi(y,z)=\sum_{k=1}^K \lambda_k \varphi_k(y)\psi_k(z)$ with $K\in \N$ and $\lambda_k\in \R$, we can generalize \eqref{etape29} by noting that
\begin{align*}
& \langle \mu_t\otimes \mu_t,\chi\rangle=\sum_{k=1}^K \lambda_k \langle \mu_t,\varphi_k\rangle \langle \mu_t,\psi_k\rangle,\\
& \widetilde{\Delta}^{(2)} \chi(y,z)=\sum_{k=1}^K \lambda_k \widetilde{\Delta}\varphi_k(y)\psi_k(z)+\sum_{k=1}^K \lambda_k \varphi_k(y)\widetilde{\Delta}\psi_k(z).
\end{align*}
Since every function $\chi$ as in (iii) can be approximated by functions of the previous form for the bounded pointwise topology, the proposition is proved.
\end{proof}

\subsection{Feynman-Kac formula in the case without interaction}

Following the approaches of \cite{bansayedelmasmarsalletran,evanssteinsaltz} in the case without interaction, it is more convenient to renormalize \eqref{def:mu_t} by $\E(\langle \bar{X}_t,1\rangle)$ instead of $\langle \bar{X}_t,1\rangle$. Let us define the random measure $\nu_t$ for a measurable function $\varphi$ on $\D_{\R^d}$ by:
\begin{equation}
\langle \nu_t,\varphi\rangle =\frac{\langle \bar{X}_t,\varphi\rangle}{\E(\langle \bar{X}_t,1\rangle)}.
\end{equation}Contrarily to $\mu_t$, the measure $\nu_t$ is not a probability measure, but its intensity measure $\E\nu_t$ is.

\begin{prop}Assume that $p\sigma^2 \, r(s,y)=1$ for the sake of simplicity. In the case without interaction, that is if $\gamma(s,y,\bar{X}^s)=\gamma(s,y)$, then we have the following Feynman-Kac representation for any test function $\varphi$ as in \eqref{testfunction:dawson}:
\begin{equation}
\langle \E\nu_t,\varphi\rangle=\frac{\E\big(\langle \bar{X}_t,\varphi\rangle \big)}{\E\big(\langle \bar{X}_t,1\rangle \big)}=\frac{\E\Big(\varphi(W_t)e^{\int_0^t \gamma(s,W_s)ds}\Big)}{\E\Big(e^{\int_0^t \gamma(s,W_s)ds}\Big)},\label{FK1}
\end{equation}where $(W_t)_{t\in \R_+}$ is a standard historical Brownian process, with values in $\Co(\R_+,\R^d)$.
\end{prop}

\begin{proof}Let us denote by $\mathbb{M}_t$ the probability measure defined by the right hand side of \eqref{FK1}. Using Itô's formula, we check that both $(\E\nu_t)_{t\in \R_+}$ and $(\mathbb{M}_t)_{t\in \R_+}$ solve the following evolution equation in $(m_t)_{t\in \R_+}$: for any test function $\varphi$ as in \eqref{testfunction:dawson},
\begin{align*}
\langle m_t,\varphi\rangle = \langle m_0,\varphi\rangle+ \frac{1}{2}\int_0^t \langle m_s,\widetilde{\Delta}\varphi(s,.)\rangle ds + \int_0^t \Big(\langle m_s,\varphi \gamma(s,.)\rangle-\langle m_s,\varphi\rangle \langle m_s,\gamma(s,.)\rangle\Big)ds,
\end{align*}which admits a unique solution.
\end{proof}

\noindent In the general case for $r(s,y)$, $p$ and $\sigma^2$, the standard historical Brownian motion has to be replaced with the path-process $W$ introduced in \eqref{def:W} and associated with the diffusion \eqref{def:Yt}. \\
The Feynman-Kac formula, in the case where there is no interaction but still possible dependence of the dynamics on the past, allows for instance easier simulation schemes for approximating quantities that can express in terms of $\E\nu_t$. The right hand side of \ref{FK1} can be approximated by Monte-Carlo methods without branching mechanism.

\subsection{Feynman-Kac formula in the logistic case}

Let us now establish a Feynman-Kac formula for the logistic case without past or trait dependence.
\begin{prop} Assume that  $r(s,y)= \bar{R}$ with $p\sigma^2 \bar{R}=1$ and  $\gamma(s,y,X)=\alpha-\eta \langle X,1\rangle$, where $\alpha,\ \eta>0$.
For any test function $\varphi$ as in \eqref{testfunction:dawson},
\begin{equation}
\langle \E\nu_t,\varphi\rangle=\frac{\E\Big(\varphi(W_t)e^{\int_0^t (\alpha N_s-\eta N_s^2) ds}\Big)}{\E\Big(e^{\int_0^t (\alpha N_s-\eta N_s^2)ds}\Big)},\label{FK2}
\end{equation}where $(W_t)_{t\in \R_+}$ is a standard historical Brownian process and  $N$ is the solution of the following equation:
\begin{equation}
N_t=\langle \bar{X}_0,1\rangle + \int_0^t \big(2\bar{R}+\alpha N_s-\eta N_s^2\big)ds+\int_0^t \sqrt{2\bar{R} N_s}\, dB_t,
\end{equation} $B$ being a standard Brownian motion independent of $W$.
\end{prop}

\begin{proof}Let $\varphi$ be a test function of the form \eqref{testfunction:dawson} and $g\in \Co^2_b(\R_+,\R)$.
From \eqref{martingalecasdawson}, we get that the process $\bar{N}_t = \langle \bar{X}_t,1\rangle$ satisfies the following martingale problem:
\begin{align}
M^1_t = & \bar{N}_t-\bar{N}_0 - \int_0^t \big(\alpha \bar{N}_s-\eta \bar{N}_s^2\big)ds\label{etape3}
\end{align} is a square integrable martingale with quadratic variation $
 \langle M^1\rangle_t=\int_0^t 2\bar{R} \bar{N}_s\ ds$. Taking the expectation leads to:
\begin{align}
\E(\bar{N}_t)= & \E(\bar{N}_0)+\int_0^t \E\big(\alpha \bar{N}_s - \eta \bar{N}_s^2\big)ds.\label{etape5}
\end{align}

Using Itô's formula with \eqref{martingalecasdawson}, \eqref{etape3}, \eqref{etape5} and taking the expectation,
\begin{align*}
\E\Big(\frac{\langle \bar{X}_t,\varphi\rangle g(\bar{N}_t)}{\E(\bar{N}_t)}\Big)= & \E\Big(\frac{\langle \bar{X}_0,\varphi\rangle g(\bar{N}_0)}{\E(\bar{N}_0)}\Big)+\int_0^t \Big[\E\Big(\frac{g(\bar{N}_s)}{\E(\bar{N}_s)} \langle \bar{X}_s,\frac{1}{2}\widetilde{\Delta}\varphi\rangle\Big)\\
+ &   \E\Big(\frac{g(\bar{N}_s)(\alpha\bar{N}_s-\eta \bar{N}_s^2)}{\E(\bar{N}_s)}\langle \bar{X}_s,\varphi\rangle\Big) - \frac{1}{\E(\bar{N}_s)^2}\E\big(g(\bar{N}_s)\langle \bar{X}_s,\varphi\rangle\big)\E\big(\alpha \bar{N}_s-\eta \bar{N}^2_s\big)\\
 + &\E\Big(\frac{\langle \bar{X}_s,\varphi\rangle}{\E(\bar{N}_s)}g'(\bar{N}_s)\big(\alpha \bar{N}_s-\eta \bar{N}_s^2\big) \Big)
 + \frac{1}{2}\E\Big(\frac{\langle \bar{X}_s,\varphi\rangle}{\E(\bar{N}_s)}g''(\bar{N}_s)\big(2\bar{R} \bar{N}_s\big)\Big)\\
 + &\frac{2}{2}\E\Big(\frac{g'(\bar{N}_s)}{\E(\bar{N}_s)}2\bar{R} \langle \bar{X}_s,\varphi\rangle \Big)\Big]ds.
\end{align*}If we define the measure
$\bar{\nu}_t$ on $\Co(\R_+,\R^d)\times \R_+$ by
$$\langle \bar{\nu}_t,\varphi\otimes g\rangle=\E\Big(\frac{\langle \bar{X}_t,\varphi\rangle g(\bar{N}_t)}{\E(\bar{N}_t)}\Big),$$then $(\bar{\nu}_t)_{t\in\R_+}$ solves the following equation
\begin{align}
\langle \bar{\nu}_t,\varphi\otimes g\rangle=& \langle \bar{\nu}_0,\varphi\otimes g\rangle+\int_0^t \Big[\Big\langle \bar{\nu}_s,\frac{1}{2}\widetilde{\Delta}\varphi\otimes g+  \varphi \otimes (g(n)(\alpha n-\eta n^2)) \nonumber\\
& + \varphi\otimes g'(n)(\alpha n -\eta n^2)+\frac{1}{2}\varphi \otimes 2\bar{R} g''(n)n+ \varphi\otimes 2\bar{R} g'\Big\rangle \nonumber\\
& -  \langle \bar{\nu}_s,\varphi \otimes g\rangle \langle \bar{\nu}_s,1 \otimes (\alpha n-\eta n^2)\rangle\Big]ds.\label{etape6}
\end{align}
Let $Y_t=\exp\big(\int_0^t (\alpha N_s-\eta N_s^2)ds\big)$. Using Itô's formula with the processes $(W_t)_{t\in \R_+}$ and $(N_t)_{t\in \R_+}$, we get
\begin{align*}
\E\Big(\frac{\varphi(W_t)g(N_t)Y_t}{\E(Y_t)}\Big)= & \E\Big(\varphi(W_0)g(N_0)\Big)+\int_0^t \Big[
\E\Big(\frac{1}{2}\widetilde{\Delta}\varphi (W_s)\frac{g(N_s)Y_s}{\E(Y_s)}\Big)\\
& +\E\Big(\frac{g(N_s)}{\E(Y_s)}\varphi(W_s)\big(\alpha N_s-\eta N_s^2\big)Y_s\Big)\\
& - \frac{1}{\E(N_s)^2}\E\Big(\varphi(W_s)g(N_s)Y_s\Big)\E\Big(\alpha N_s-\eta N_s^2\Big)\\
& +\E\Big(\frac{\varphi(W_s)g'(N_s)Y_s}{\E(Y_s)}\big(2\bar{R}+\alpha N_s-\eta N_s^2\big)\Big)+\E\Big(\frac{1}{2}\frac{\varphi(W_s)g''(N_s)Y_s}{\E(Y_s)}2\bar{R}N_s\Big)
\Big]ds.
\end{align*}
The measure $\widetilde{\nu}_t$ on $\Co(\R_+,\R^d)\times \R_+$ defined by
$$\langle \widetilde{\nu}_t,\varphi\otimes g\rangle=\E\Big(\frac{\varphi(W_t)g(N_t)Y_t}{\E(Y_t)}\Big)$$is hence another solution of \eqref{etape6}. The uniqueness of the solution of this equation provides the announced result.
\end{proof}

\bigskip
\noindent \textbf{Acknowledgements:} The authors thank the Associate Editor for his constructive comments. This work benefited from the support of the  ANR MANEGE (ANR-09-BLAN-0215), from the Chair ``Modélisation Mathématique et Biodiversité" of Veolia Environnement-Ecole Polytechnique-Museum National d'Histoire Naturelle-Fondation X. It also benefited from the visit of both authors to the Institute for Mathematical Sciences of the National University of Singapore in 2011.

\appendix

\section{Properties of the test functions \eqref{testfunction} and \eqref{testfunction:dawson}}\label{appendixA}

We begin with a lemma that will be useful to link the classes of test functions \eqref{testfunction} and \eqref{testfunction:dawson}. This lemma also shows that the class of test functions \eqref{testfunction} separates points.

\begin{lemme}\label{lemme-passage1}For $q\in \N^*$, recall that we denote by $k^q(u)$ the density of the Gaussian distribution with mean 0 and variance $1/q$. Let $g\in \Co_b(\R^d,\R)$. For $G(x)=x$ and $g_q(s,x)=k^q(t-s)g(x)$, we have for all $y\in \D_{\R^d}$ and all $t\in [0,T]$ at which $y$ is continuous that:
$$\lim_{q\rightarrow+\infty}G_{g_q}(y)=g(y_t).$$
\end{lemme}
\begin{proof}First notice that all the $G_{g_q}$ are bounded by $\|g\|_\infty$. Let $\varepsilon>0$. Since $y$ is continuous at $t$, so is $g\circ y$ and there exists $\alpha>0$ sufficiently small so that for every $s\in (t-\alpha,t+\alpha)$, $|g(y_s)-g(y_t)|\leq \varepsilon/2$. We can then choose $q$ sufficiently large such that
$$\int_{|t-s|>\alpha} k^q(t-s)ds <\frac{\varepsilon}{4\|g\|_\infty}.$$
Then:
\begin{align*}
|G_{g_q}(y)-g(y_t)|= & \Big|\int_0^T k^q(t-s)\big(g(y_s)-g(y_t)\big)ds \Big|\\
\leq & 2 \|g\|_\infty \varepsilon \int_{|s-t|\geq \alpha} k^q(t-s) ds + \frac{\varepsilon}{2} \int_{|s-t|<\alpha}k^q(t-s)ds\leq \varepsilon.
\end{align*}
\end{proof}

We are now in position to give the:

\begin{proof}[Proof of Lemma \ref{lien:fonctionstest}]
We can assume without restriction that the functions $g_j$ in the definition of $\varphi$ \eqref{theoremjakubowskihistorique} are positive. Let us define for $y\in \Co(\R_+,\R^d)$:
\begin{align*}
\varphi_q(y)=\exp \Big( \int_0^T \sum_{j=1}^m \log g_j(y_s) k^q(t_j-s) ds\Big).
\end{align*}By Lemma \ref{lemme-passage1}, the term in the integral is bounded uniformly in $q$ and $y$ and converges when $q$ tends to infinity to
$\sum_{j=1}^{m} \log g_j(y_{t_j}).$ As a consequence, for every $y\in \Co(\R_+,\R^d)$, the sequence $(\varphi_q(y))_{n\in \N^*}$ is bounded and converges to
$$\exp\Big(\sum_{j=1}^{m} \log g_j(y_{t_j})\Big)=\prod_{j=1}^{m}g_j(y_{t_j})=\varphi(y)$$when $q$ tends to infinity. Moreover,
\begin{align*}
\mathcal{D}^2\varphi_q(t,y)= & \exp\Big(\int_0^T \sum_{j=1}^m \log g_j(y_s) k^q(t_j-s) ds\Big)\\
 & \times \Big(\int_t^T \sum_{j=1}^m \Delta_x (\log g_j)(y_t) k^q(t_j-s)ds+\sum_{i=1}^d \Big(\int_t^T \sum_{j=1}^m \partial_{x_i} (\log g_j)(y_t) k^q(t_j-s)ds\Big)^2\Big).
\end{align*}When $q$ tends to infinity, we have by Lemma \ref{lemme-passage1}:
\begin{align*}
\lim_{q\rightarrow +\infty} \mathcal{D}^2\varphi_q(t,y)= &  \prod_{j=1}^m g_j(y_{t_j}) \Big( \sum_{j\ |\ t_j>t} \frac{\Delta_x g_j(y_t)}{g_j(y_t)}-\sum_{i=1}^d \frac{\big(\partial_{x_i} g_j(y_t)\big)^2}{g^2_j(y_t)} +\sum_{i=1}^d \Big(\sum_{j\ |\ t_j>t} \frac{\partial_{x_i} g_j(y_t)}{g_j(y_t)}\Big)^2\Big)\\
= &  \prod_{j=1}^m g_j(y_{t_j}) \Big( \sum_{j\ |\ t_j>t} \frac{\Delta_x g_j(y_t)}{g_j(y_t)}+2 \sum_{i=1}^d \sum_{\substack{j\not= k\\ t_j,t_k>t}}\frac{\partial_{x_i} g_j(y_t) \partial_{x_i} g_k(y_t)}{g_j(y_t)g_k(y_t)}\Big)\\
= &  \prod_{j=1}^m g_j(y_{t_j}) \frac{\Delta_x\big(\prod_{j\ |\ t_j>t}g_j\big)(y_t)}{\prod_{j\ |\ t_j>t}g_j(y_t)}=\widetilde{\Delta}\Big(\prod_{j=1}^m g_j\Big)(t,y).
\end{align*}This concludes the proof.

\end{proof}

\section{Technical result on concatenated paths}\label{section:concat}

Let $\H$ be the set of increasing bijections from $[0,T]$ to $[0,T]$, where $T>0$.
We recall that the Skorokhod distance is defined for $y,z\in \D$ by:
\begin{equation}
\dSk(y,z)=\inf_{\lambda\in \H}\max\left\{\|y\circ \lambda -z\|_\infty, \sup_{t,s<T}\Big|\log\Big(\frac{\lambda(t)-\lambda(s)}{t-s}\Big)\Big|\right\}.\label{dsk}
\end{equation}

In the sequel, we consider $y, \ z\in \D$ and $s,r\in [0,T]$. Without loss of generality, we can assume that $s<r$.

\begin{prop}\label{prop_Sk}
If $\dSk(y,z)<\varepsilon$ and if $s$ and $r$ are sufficiently close so that:
\begin{equation}
0\leq \max \Big\{\log\frac{r}{s}, \log\frac{T-s}{T-r}\Big\} \leq \varepsilon.\label{condition}
\end{equation}Then for all $w\in \D$,
$\dSk((y|s|w),(z|r|w))<3\varepsilon$.
\end{prop}

In the proof, we will need the following change of time $\lambda_0\in \H$:
\begin{equation}
\lambda_0(u)=\frac{r}{s}u\ind_{u\leq s}+\Big(r+\frac{T-r}{T-s}(u-s)\Big)\ind_{u>s}.\label{lambda0}
\end{equation}
The bijection $\lambda_0$ depends on $r$ and $s$. For $u$ and $v\in [0,T]$, we have:
\begin{align}
\Big|\log\Big(\frac{\lambda_0(u)-\lambda_0(v)}{u-v}\Big)\Big|\leq \max\Big\{\log\Big(\frac{r}{s}\Big),\log\Big(\frac{T-s}{T-r}\Big)\Big\}.
\end{align}The right hand side converges to 0 when $r/s$ converges to 1, and is upper bounded by $\varepsilon$ under the Assumptions \eqref{condition} of Proposition \ref{prop_Sk}.

\begin{lemme}\label{lemme1}For all $w\in \D$. If \eqref{condition} is satisfied, then $\dSk(w\circ \lambda_0, w)\leq \varepsilon.$
\end{lemme}
\begin{proof}The infimum in \eqref{dsk} can be upper bounded by choosing $\lambda=\lambda_0^{-1}$, which is the inverse bijection of $\lambda_0$:
\begin{equation}
\lambda^{-1}_0(u)=\frac{s}{r}u\ind_{u\leq r}+\Big(s+\frac{T-s}{T-r}(u-r)\Big)\ind_{u>r}.\label{lambda0-1}
\end{equation}
For such choice, we have:
\begin{align*}
\dSk(w\circ \lambda_0, w)\leq & \max\Big\{0,\max\Big(\log\frac{r}{s},\log \frac{T-s}{T-r}\Big)\Big\} \leq \varepsilon.
\end{align*}
\end{proof}

Let us now prove Proposition \ref{prop_Sk}:
\begin{proof}
By the triangular inequality:
\begin{align}
& \dSk((y|s|w),(z|r|w))\leq  A+B,\qquad \mbox{ where }\\
& A= d((y|s|w),(y\circ \lambda_0 |r| w\circ \lambda_0)) \nonumber\\
& B= d((y\circ \lambda_0 |r| w\circ \lambda_0),(z|r|w)).\nonumber
\end{align}By Lemma \ref{lemme1}, $A\leq \varepsilon$. For the second term, using Lemma \ref{lemme1} again:
\begin{align}
B\leq & \dSk(y\circ \lambda_0,z)+d(w\circ \lambda_0,w)\leq  \dSk(y\circ \lambda_0,z)+\varepsilon.
\end{align}
Now, since $\dSk(y,z)\leq \varepsilon$, there exists $\lambda\in \H$ such that $\|y\circ \lambda-z\|_\infty\leq 2\varepsilon$ and
$$\sup_{u,v\leq T}\Big|\log \frac{\lambda(u)-\lambda(v)}{u-v}\Big|\leq 2\varepsilon.$$Then, considering the change of time $\lambda_0^{-1}\circ \lambda$:
\begin{align*}
\lefteqn{\dSk(y\circ \lambda_0,z)\leq  \max\Big\{\|y\circ \lambda_0\circ \lambda_0^{-1}\circ \lambda - z\|_\infty, \ \sup_{u,v\leq T}\Big|\log \frac{\lambda_0^{-1}\circ \lambda(u)-\lambda_0^{-1}\circ \lambda(v)}{u-v}\Big|\Big\}}\\
\leq & \max\Big\{\|y\circ \lambda-z\|_\infty, \ \sup_{u,v\leq T}\Big|\log \frac{\lambda_0^{-1}\circ \lambda(u)-\lambda_0^{-1}\circ \lambda(v)}{\lambda(u)-\lambda(v)}\Big| + \sup_{u,v\leq T}\Big|\log \frac{\lambda(u)-\lambda(v)}{u-v}\Big| \Big\}\\
\leq & \max\Big(\varepsilon,3\varepsilon\Big).
\end{align*}
\end{proof}



{\footnotesize
\providecommand{\noopsort}[1]{}\providecommand{\noopsort}[1]{}\providecommand{\noopsort}[1]{}\providecommand{\noopsort}[1]{}

}

\end{document}